\documentclass[12pt]{amsart}
\usepackage[english]{babel}
\usepackage{a4wide,color,graphicx,amsmath,amssymb,verbatim,mathrsfs,latexsym,pxfonts,bm}
\allowdisplaybreaks
\newtheorem{theo}{Theorem}

\newtheorem{lemma}{Lemma}

\newtheorem{rem}{Remark}

\newcommand{\R}{\mathbb{R}}	
	
\newcommand{\N}{\mathbb{N}}

\newcommand{\eps}{\varepsilon}	
\newcommand{\pa}{\partial}		
\newcommand{\Div}{\textrm{div}\,}

\newcommand{\na}{\nabla}

\newcommand{\chf}[1]{{\raisebox{3pt}{\Large $\chi$}}_{#1}}
\newcommand{\rhoed}{\rho^{(\eps,\delta)}}
\newcommand{\rhod}{\rho^{(\delta)}}

\newcommand{\avg}[1]{\frac{1}{|\Omega|}\int_{\Omega} #1\, dx}

\title[]{Connection between a degenerate particle flow model and a free boundary problem}

\author{Li Chen}
\address{University of Mannheim, Department of Mathematics, 68131 Mannheim, Germany}
\email{chen@math.uni-mannheim.de}

\author{Simone G\"ottlich}
\address{University of Mannheim, Department of Mathematics, 68131 Mannheim, Germany}
\email{goettlich@uni-mannheim.de}

\author{Nicola Zamponi}
\address{Ulm University, Institute of Applied Analysis, Helmholtzstra\ss e 18, 89081 Ulm}
\email{nicola.zamponi@uni-ulm.de}

\date{\today}

\thanks{The first author acknowledges partial support from the German Research Foundation (DFG), grant CH 955/8-1. The second author acknowledges financial support from the German Research Foundation (DFG), grants GO1920/11-1 and GO1920/12-1.}

\begin{document}

\begin{abstract}
In this paper a strongly degenerate parabolic equation derived from a density dependent particle flow model is studied. Furthermore, a free boundary problem and its connection to the strongly degenerate parabolic equation is investigated. 
First, it is shown that the strongly degenerate parabolic equation
has a unique global bounded weak solution that converges towards a steady state for large time horizons. 
Two scenarios might occur: When the average density $\rho_{\infty}$ is larger than a certain critical density $\rho_{cr}$, the steady state coincides with $\rho_{\infty}$ and the convergence rate is exponential in the $L^2$ norm; while in the opposite case $\rho_{\infty}<\rho_{cr}$, the steady state is unknown and the convergence is algebraic in a negative Sobolev seminorm. 
Further investigations show that for radially symmetric and decreasing initial data, the solution of the strongly degenerate parabolic equation can be constructed by using the solution of a corresponding free boundary problem. Moreover, the global existence of weak solutions to the latter problem is proved.
Finally, numerical experiments in two space dimensions are presented, which show that segregation phenomena can appear when the initial average density is smaller than the critical density.\\[1ex]
\end{abstract}

\maketitle

\section{Introduction}

Motivated by a newly developed density dependent particle flow model in \cite{WGA2021} with applications to material flow problems and swarm behavior, we consider the strongly degenerate parabolic equation with non-flux boundary condition 
\begin{align}\label{1}
	&\pa_t\rho = \Delta f(\rho) \quad\mbox{on } Q_T\equiv \Omega\times (0,T],\\
\label{2}	&\nabla f(\rho)\cdot\nu = 0\quad\mbox{on }\pa\Omega\times (0,T),
	\quad\rho(\cdot,0) = \rho_0\quad\mbox{on }\Omega,
\end{align}
where $f : [0,\infty)\to[0,\infty)$ satisfies:
\begin{align}\label{hp.f}
	f\in C^1([0,\infty)),\quad f'>0~~\mbox{on }(\rho_{cr},\infty),\quad f = 0~~\mbox{on }[0,\rho_{cr}],
\end{align}
and $\rho_{cr}>0$ is the given critical density 
and $\Omega$ is an open, bounded domain in $\R^d$, $d\geq 2$.
We also use the notation $Q\equiv \Omega\times (0,\infty)$.

In the case of radially symmetric initial data which are decreasing in the radial coordinate, i.e. $\rho_0(x)=\tilde\rho_0(|x|)$, we will also establish a connection between \eqref{1}--\eqref{2} and the following free boundary problem 
which only degenerates on the free boundary:
\begin{align}\label{fb}
&\pa_t u = \Delta f(u)\qquad \mbox{ for }\quad |x|<R(t), ~~ t > 0,\\
\label{fb.bc}
&u = \rho_{cr}, \qquad
(\rho_{cr} - \rho_0)\pa_t R(t) = |\nabla f(u)|
\qquad \mbox{ for }\quad |x|=R(t), ~~ t>0,\\
\label{fb.ic}
&u(x,0) = \rho_0(x)\quad \mbox{ for } \quad |x|<R_0,
\qquad \mbox{ where }\quad R(0) = R_0 := \tilde{\rho}_0^{-1}(\rho_{cr})>0.
\end{align}
The unknowns of \eqref{fb}--\eqref{fb.ic} are the functions $u : \mathcal{G}_T\to[0,\infty)$ and $R : [0,\infty)\to[0,\infty)$, where we define
\begin{align*}
\mathcal{G}_T := \left\{ 
(x,t)\in\Omega\times (0,T)~:~ |x|<R(t)\right\}\subset Q_T.
\end{align*}
Precisely, we will show that through the solution of the free boundary problem \eqref{fb}--\eqref{fb.ic}, whose existence is also going to be proved, a solution to the strongly degenerate parabolic equation \eqref{1}--\eqref{2} can be constructed.

Strongly degenerate parabolic equations and systems have appeared in the literature of sedimentation related modeling with different boundary conditions \cite{BCBT1999,CBB1996}. A classical scalar degenerate parabolic equation has the following form
\begin{align}\label{ph}
	\pa_t\rho + \nabla\cdot (\Phi(\rho)) = \Delta f(\rho),
\end{align}
where $f$ is a given non-decreasing Lipschitz continuous function with $f(0)=0$ and $\Phi$ is a given density dependent flux function. It is well-known that due to the strong degeneracy, this type of equations exhibit many hyperbolic properties such as the formation of shocks and hence the non-uniqueness of weak solutions. For example, it is proved in \cite{Vol2000} that a weak solution is unique if it satisfies the entropy condition. Another result for an entropy solution in multi-dimension has been obtained in \cite{Ca1999}. For more results on entropy solutions we refer to \cite{BV2006, CK2005, CP2003, VH1969}. The well-posedness theory of weak entropy solutions to strongly degenerate parabolic equation on Riemannian manifolds is established in \cite{GKM2017}. For the general theory of weak solutions to degenerate parabolic equation we refer to \cite{WZYL_Book,Z1989}, and to \cite[Section 5]{TKBB2003} for a short review of the development of the mathematical theory in this direction. Recently, a density dependent particle flow model has been formally derived through a mean-field limit approach in \cite{WGA2021} and is given by
\begin{align}\label{mf}
	\pa_t\rho + \nabla\cdot (\rho v_T - k(\rho)\nabla\rho)=0,
\end{align}
where $k(\rho)=\rho H(\rho-\rho_{cr})$ with $\rho_{cr}>0$ the critical density and $H$ is the Heaviside function. 
Originally, equation~\eqref{mf} has been employed as a material flow model to describe the transport of identical homogeneous particles on a conveyor belt with velocity $v_T$ subject to a nonlinear interaction forces \cite{GHSSV2014}. It has also been used to describe the dynamical behavior of pedestrian crowds \cite{GKS2018,HM1995} as well as biological swarms subject to both attractive and repulsive interaction forces \cite{COP2009}.
Unlike the parabolic-hyperbolic equation \eqref{ph}, the transport part of equation \eqref{mf} is linear and the particle flow model has no hyperbolic structure. 
Therefore, we study this equation without transport term in a first step and deal with \eqref{1}--\eqref{2}. We also give the regularity assumption on $f$ as stated \eqref{hp.f} instead of taking the Heaviside function, for technical simplicity.

Due to the parabolic structure of \eqref{1}--\eqref{2}, the results for existence and uniqueness of weak solution is not surprising, but also not trivial, given the presence of a strong degeneracy.
\begin{theo} [Existence and uniqueness of the weak solution] \label{thm.existence} Assume that \eqref{hp.f} holds and the initial data $\rho_0\in L^{\infty}(\Omega)$. Then the initial boundary value problem \eqref{1}--\eqref{2} admits a unique global-in-time weak solution $\rho$ satisfying
	\begin{align}\label{weaksol.1}
		\rho\in L^\infty(Q),\quad \nabla f(\rho)\in L^2(Q),\quad \pa_t\rho\in L^2(0,\infty; H^{1}(\Omega)'),\\
		\label{weaksol.2}
		-\int_0^T\int_\Omega\rho\pa_t\phi dx dt - \int_{\Omega}\rho_{0}\phi(0)dx + 
		\int_0^T\int_{\Omega}\nabla f(\rho)\cdot\nabla\phi dx dt = 0\\
		\nonumber
		\forall\phi\in C^\infty_c(\Omega\times [0,T)),\quad\forall T>0
	\end{align}
and such that, for a.e.~$x\in\Omega$, the mapping $t\in (0,\infty)\mapsto \min(\rho(x,t),\rho_{cr})\in\R$ is non-decreasing.
\end{theo}
The existence of global-in-time weak solutions is shown by 
regularizing the equation with a Laplacian term and
employing a compactness argument based on the Div-Curl Lemma, while the uniqueness of weak solutions is proved via a duality argument which exploits the monotonicity of $f$.

The next contribution of this paper is the investigation of the long-time behavior of solutions to \eqref{1}--\eqref{2}. We obtain the following result and show some numerical experiments, mostly concerned with the subcritical case.

Let us preliminary define the (relative) free energy density
\begin{align}\label{rel.entr}
F(s\vert\rho_\infty) := F(s) - F(\rho_\infty) - f(\rho_\infty)(s - \rho_\infty),\quad
F(s) := \int_0^s f(s')ds',\quad s\geq 0.
\end{align}
In some situations we will make the following additional assumption on $f$:
\begin{align}\label{hp.kappa}
\exists\kappa > 1:~~\forall R>0,~~\exists C_R, c_R > 0:\quad
c_R(s-\rho_{cr})^\kappa \leq f(s)\leq C_R (s-\rho_{cr})^\kappa,\quad \rho_{cr}\leq s\leq R.
\end{align}
We also define the following negative homogeneous Sobolev seminorm
\begin{align*}
\forall\varphi\in C^0_c(\Omega)\qquad
\|\varphi\|_{\dot{H}^1(\Omega)'} = \|\na V_\varphi\|_{L^2(\Omega)},\qquad
\begin{cases}
-\Delta V_\varphi = \varphi - \fint_\Omega\varphi dx\quad &\mbox{in }\Omega,\\
\nu\cdot\na V_\varphi = 0 \quad & \mbox{on }\pa\Omega,\\
\int_\Omega V_\varphi dx = 0 & 
\end{cases}
\end{align*}

\begin{theo}[Long-time behavior]\label{thm.ltb.diff}
Assume that \eqref{hp.f} holds and that $\Omega$ is connected.

Let $\rho$ be a solution of \eqref{1}--\eqref{2} and define 
$\rho_\infty\equiv\frac{1}{|\Omega|}\int_\Omega\rho_0 dx$.\medskip\\
%
%
If $\rho_\infty > \rho_{cr}$ (supercritical case), then 
\begin{align}\label{cnv.super}
&\rho(t)\to\rho_\infty \quad \mbox{ as } t\to\infty \mbox{ strongly in } L^p(\Omega),\quad  \forall 1\leq p < \infty,\\
\label{lt.exp}
&\exists\lambda>0:\quad
\int_\Omega F(\rho(t)\vert\rho_\infty)dx \leq e^{-\lambda t}\int_\Omega F(\rho_0\vert\rho_\infty)dx\qquad t>0.
\end{align}\smallskip\\
%
%
If $\rho_\infty = \rho_{cr}$ (critical case), then \eqref{cnv.super} holds. \medskip\\
%
%
If $\rho_\infty < \rho_{cr}$ (subcritical case), then
\begin{itemize}
\item[(i)] a function $\hat{\rho}\in L^\infty(\Omega)$ exists such that 
\begin{align}\label{cnv.rhohat}
\rho(t)\to\hat\rho\quad \mbox{ as } t\to\infty \mbox{ strongly in } L^p(\Omega),\quad  \forall 1\leq p < \infty,
\end{align}
and it holds $\hat{\rho}\leq\rho_{cr}$ a.e.~in $\Omega$.\smallskip

\item[(ii)] assuming that \eqref{hp.kappa} is fulfilled, one obtains the following algebraic convergence rate
\begin{align*}
\|(\rho(t)-\rho_{cr})_+\|_{L^1(\Omega)} + 
\|\rho(t)-\hat{\rho}\|_{\dot{H}^{1}(\Omega)'}\leq C\, t^{-\frac{1}{\kappa - 1}}
\end{align*}
for $t\to\infty$.\smallskip
\item[(iii)] If $\mbox{meas}(\{ \rho_0 > \rho_{cr} \}) > 0$, then
there exists $T^*>0$ such that 
$\int_0^{T^*}\int_\Omega |\nabla\rho|^2 dx dt = \infty$.
\end{itemize}
\end{theo}
The long-time behavior of the solutions in the supercritical case $\rho_\infty> \rho_{cr}$ is obtained from a free energy identity and a Poincar\'e-Wirtinger-type inequality (see Lemma \ref{lem.new.f} in the Appendix).
In the subcritical regime $\rho_\infty<\rho_{cr}$, the aforementioned tools
only provide a partial information on the long-time behavior of the solution. They are therefore complemented by a non-constructive argument, based on a monotonicity property of the solution, which yields the convergence of the solution towards a steady state.

The algebraic rate of convergence with respect to the negative Sobolev seminorm follows from a duality argument and assumption \eqref{hp.kappa}, which describes the behavior of $f$ near the critical value. We wish to point out that the obtained algebraic decay rate resembles closely the result holding for the porous medium equation in the whole space \cite{FriKam1980}, which is in agreement with the intuition that the dynamics of \eqref{1} under assumption \eqref{hp.kappa} should resemble a porous medium equation in the region close to the critical value. Finally, in the critical case $\rho_\infty = \rho_{cr}$ we are not able to show a suitable Poincar\'e-Wirtinger-type inequality and thus we cannot derive a decay estimate for the solution. However, on the other hand, we can deduce strong $L^p(\Omega)$ convergence towards the steady state for any $p<\infty$ by exploiting the free energy estimate and mass conservation.

The final result we present is concerned with the free boundary problem
\eqref{fb}--\eqref{fb.ic} and its relation with the degenerate parabolic equation \eqref{1}--\eqref{2}.

Due to the strong degeneracy of $f$ in \eqref{1}, if the initial data is partly larger than the critical density $\rho_{cr}$ and partly smaller, one can observe that the hyper-surface where $\rho(x,t)=\rho_{cr}$ moves continuously in time. Therefore it is interesting to study the dynamics of this hyper-surface that is mathematically a free boundary problem. Unlike the classical boundary value problem, the challenging point of free boundary problem is ''the fact that the domains are a priori unknown and have to be determined as a part of the problem'', see \cite{CSV2015}. There exist a large number of scientific works to study free boundary problems from a partial differential equation point of view. Mathematical results have been started in 1970's from the famous {\it Stefan problem} which describes the dynamics of two phase flow, see \cite{F1968,FK1975,KN1978}. There are also works on minimal surface \cite{S1984}, fractional Laplacian \cite{CSS2008}, parabolic phase transition problems \cite{ACS1996,CPS2004}, and Stefan problem with Fisher-KPP sources \cite{D2013,DG2012,DMW2014}, to name a few. It is out of the scope of this paper to make a full review on free boundary problems, we refer to \cite{CSV2015} for the current and future development of this hard research topic. Works much closer to the topic of this paper are free boundary problems on degenerate parabolic equations \cite{DS1983,CF2000,BFK2002}, for example the porous media type diffusion in \cite{Wu1984,ZL1990}, $p$-Laplacian type diffusion in \cite{Zhao1997,HS2002}, and fully nonlinear parabolic equation \cite{FS2015}. To our knowledge, there is no work on the free boundary problem within the strongly degenerate parabolic setting given by \eqref{1}. 

Since it is usually very hard to check the movement of the free boundary, in this paper we consider a radially symmetric initial data which is decreasing in radius, and try to follow the dynamics of the radius where the diffusion effect disappears. A formal computation implies the setting of the free boundary problem \eqref{fb}--\eqref{fb.ic}. In the following theorem, we prove the existence of weak solutions of this problem, from which we obtain the weak solution of the original initial boundary value problem \eqref{1}--\eqref{2}. This shows simultaneously also the long-time behavior of the solution in the subcritical case $\rho_\infty < \rho_{cr}$.

\begin{theo}[Free-boundary problem]\label{thm.fb}
Let $\Omega\subset\R^d$ be a ball of center 0. 
Assume that $\rho_0\in C^0(\overline{\Omega})$ satisfies
\begin{align*}
\rho_0(x) = \overline{\rho}_0(|x|)\quad x\in\Omega,\qquad {\overline{\rho}_0\,'<0\mbox{ in }(0,\infty)},\qquad
\rho_0(0) > \rho_{cr}.
\end{align*}
{\bf Existence.}
There exists a weak solution $(u,R)$ to \eqref{fb}--\eqref{fb.ic} satisfying 
\begin{align}\label{fb.reg}
u\in L^\infty(\mathcal{G}_T),\quad  
\sup_{t\in [0,T]}\int_{|x|<R(t)}|\na_x f(u)|^2 dx<\infty,\quad
\int_0^T\left( \int_{|x|=R(t)}|\na_x f(u)|^{3/2}d\sigma_x \right)^2 dt<\infty,\\
\label{fb.reg.2}
R\in C^{0,\alpha}([0,T]) \quad 0<\alpha<\frac{1}{3},\qquad
R'\in L^p(0,T)\quad 1<p<\frac{3}{2},
\end{align}
and the following weak formulation
\begin{align}\label{fb.weak.1}
&\int_0^T\int_{|x|<R(t)} u\pa_t\phi\, dx dt
+\int_{|x|<R_0}\rho_0(x) \phi(x,0)dx \\ \nonumber
&\qquad =\int_0^T\int_{|x|<R(t)}\na\phi\cdot\na f(u)\, dx dt
+\int_0^T\int_{|x|=R(t)}\phi |\na f(u)|d\sigma dt\\ \nonumber
&\forall \phi\in C^1(\overline{\mathcal{G}_T}),~~\phi(\cdot,T)\equiv 0,\\
& R(t) = R(0) + \int_0^t\frac{|\na_x f(u(\tau,x))|}{\rho_{cr} - \rho_0(x)}\Big\vert_{|x|=R(\tau)}d\tau,\quad 0\leq t\leq T.
\end{align}
The function $u$ is also radially symmetric and it holds
\begin{align*}
|\na f(u)| = -\nu\cdot\na f(u)\quad\mbox{on $|x|=R(t)$, $t>0$.}
\end{align*}
{\bf Relation to \eqref{1}--\eqref{2}.}
The function $\rho : \Omega\times (0,T)\to [0,\infty)$ defined as
\begin{align}\label{rho.fb}
\rho(t) = u(t)\chf{{E(t)}} + \rho_{0}\chf{\Omega\backslash {E(t)}}\quad t > 0,\quad 
E(t) = \{|x|\leq R(t)\}
\end{align}
is a weak solution to \eqref{1}--\eqref{2} according to Theorem~\ref{thm.existence}. Furthermore, 
if $\rho_\infty \leq \rho_{cr}$,  
the steady state $\hat{\rho} = \lim_{t\to\infty}\rho(t)$ to \eqref{1}--\eqref{2} is given by
\begin{align}\label{SteadyStateRepr}
\hat{\rho} = \rho_{cr}\chf{{E_\infty}} + \rho_{0}\chf{\Omega\backslash {E_\infty}},\quad E_\infty = \{|x|\leq R_\infty\},\quad R_\infty = \lim_{t\to\infty}R(t).
\end{align}
\end{theo}

\begin{rem} Let $\hat R$ be the radius of $\Omega$. Then clearly $R(t)\leq \hat R$ for every $t>0$, as well as $R_\infty\leq \hat R$. Moreover, mass conservation implies that $R_\infty = \hat R$ if $\rho_\infty\geq \rho_{cr}$, 
while $R_\infty < \hat R$ if $\rho_\infty < \rho_{cr}$. The latter case \eqref{SteadyStateRepr} implies that $\hat{\rho}$ is a discontinuous function as $E_\infty\subsetneqq \Omega$.
\end{rem}

The existence of solutions to \eqref{fb}--\eqref{fb.ic} is obtained by 
applying a change of variables which transforms it into a semilinear equation on a constant domain (a ball) with time-dependent coefficients coupled with an ordinary differential equation (ODE). From this latter system an approximated problem is then considered by adding a truncation operator with parameter $\eps>0$ and a time regularization with parameter $\delta>0$ to the equations. The limits $\delta\to 0$, $\eps\to 0$ are performed (in this order) by deriving suitable uniform estimates for the approximated solution and by employing Aubin-Lion's Lemma.

A remarkable aspect of the result contained in Theorem \ref{thm.fb} is the fact that the solution $\rho$ to \eqref{1}--\eqref{2} is {\em discontinuous} along the critical set $\pa E(t) = \{ |x|=R(t) \}$, as the density remains above the critical value $\rho_{cr}$ inside $E(t)=\{|x|\leq R(t)\}$ and stays below $\rho_{cr}$ outside $E(t)$. In fact, the evolution of $\rho(t)$ is driven by diffusion inside $E(t)$, which pushes $\rho(t)$ downwards towards the critical value $\rho_{cr}$. As a consequence of mass conservation, the domain $E(t)$ expands, i.e.~$R(t)$ grows with time. On the other hand, outside $E(t)$ the solution $\rho(t)$ coincides with the initial datum since the equation \eqref{1} reduces to $\pa_t\rho(t)=0$ in $\Omega\backslash\overline{E(t)}$. As a consequence of these combined phenomena, the solution is discontinuous across the border of $E(t)$ for every time $t>0$. Notice the agreement between these observations and statement (iii) for the case $\rho_\infty<\rho_{cr}$ in Theorem~\ref{thm.ltb.diff}: as a matter of fact, the hypothesis that the set $\{\rho_{0}>\rho_{cr}\}$ has positive measure has been also assumed in the proof of Theorem \ref{thm.fb}.

The arrangement of the paper is the following. In Section \ref{sec.ex} we prove Theorem~\ref{thm.existence}. In Section \ref{sec.longtime} we prove Theorem~\ref{thm.ltb.diff}. In Section \ref{sec.fb} we prove Theorem\ref{thm.fb}. 
Finally, in Section \ref{sec.numerics}, several numerical experiments for the solution behavior are presented, showing the convergence towards the constant steady state in the supercritical case (which is expected from the analysis), as well as a segregation phenomenon for initial data with subcritical mass. 

\section{Existence and uniqueness analysis (Proof of theorem \ref{thm.existence})}\label{sec.ex}
In this section, we prove Theorem \ref{thm.existence}. 
The proof of the existence starts from the construction of a family of approximated solutions depending on two approximation parameters $\eps$ and $\delta$. We will first show that the approximated solutions satisfy estimates that are uniform with respect to the approximation parameters. Then, we will invoke the Div-Curl Lemma to deduce the existence of a strongly convergent subsequence of approximated solutions and show that the limit of such a subsequence is a global weak solution to \eqref{1} satisfying 
\eqref{weaksol.1}--\eqref{weaksol.2}. In the end, we prove the uniqueness by using the monotonicity property of $f$ together with the $H^{-1}$ method.\medskip\\
{\bf Approximated solutions.} We introduce two approximation parameters: $\delta>0$ (lower order regularization) and $\eps>0$ (higher order regularization).

For $0<\delta<\min\{\|\rho_{0}\|_\infty^{-1},\rho_{cr}^{-1}\}$ we define
\begin{align*}
f_\delta(r) = \begin{cases}
f(r)  & r < \delta^{-1};\\
f(\delta^{-1}) + f'(\delta^{-1})(r-\delta^{-1}) & r\geq \delta^{-1}.
\end{cases}
\end{align*}
Let us consider the approximated problem (with $m > d/2$ a given integer)
\begin{align}\label{app.1}
&\int_0^T\langle \pa_t\rhoed , \phi\rangle dt + \eps\int_0^T (\rhoed,\phi)_{H^m(\Omega)} dt
+ \delta\int_0^T (\rhoed,\phi)_{H^1} dt\\ \nonumber
& \quad= -\int_0^T\int_{\Omega}\nabla f_\delta(\rhoed)\cdot\nabla\phi dx dt\qquad
\forall\phi\in L^2(0,T; H^m(\Omega)),\\
&\label{app.1.ic}
\rhoed(0) = \rho_0\quad\mbox{in }\Omega,
\end{align}
to be solved for $\rhoed\in L^2(0,T; H^m(\Omega))$. We reformulate \eqref{app.1}--\eqref{app.1.ic} as a fix point problem for the mapping $\mathcal{T} : (u,\sigma)\in L^2(0,T; H^1(\Omega))\times [0,1]\mapsto \rho\in L^2(0,T; H^1(\Omega))$,
\begin{align}\label{fixp.1}
&\int_0^T\langle \pa_t\rho , \phi\rangle dt + \eps\int_0^T (\rho,\phi)_{H^m(\Omega)} dt
+ \delta\int_0^T (\rho,\phi)_{H^1} dt  \\ \nonumber
&\quad =-\sigma\int_0^T\int_{\Omega}\nabla f_\delta(u)\cdot\nabla\phi dx dt\qquad
\forall\phi\in L^2(0,T; H^m(\Omega)),\\
&\label{fixp.1.ic}
\rho(0) = \rho_0\quad\mbox{in }\Omega.
\end{align}
The mapping $\mathcal T$ is clearly well-defined given the assumptions on $f$ and the definition of $f_\delta$. Furthermore $\mathcal{T}(\cdot,0)$ is constant. Also choosing $\phi=\rho$ in \eqref{fixp.1} yields the bound
\begin{align}\label{fixp.2}
\|\rho\|_{L^\infty(0,T; L^2(\Omega))} +
\|\pa_t\rho\|_{L^2(0,T; H^{m}(\Omega)')}
+ \|\rho\|_{L^2(0,T; H^m(\Omega))}\leq C(\eps,\delta)( \|\rho_0\|_2 + \|u\|_{L^2(0,T; H^1(\Omega))} ).
\end{align}
This means that $\mathcal{T}(L^2(0,T; H^1(\Omega))\times [0,1])$ is bounded in the space
\begin{align*}
\left\{
\rho\in L^2(0,T; H^{m}(\Omega))\quad\vert\quad 
\pa_t\rho\in L^2(0,T; H^{m}(\Omega)')
\right\}
\end{align*}
which, due to Aubin-Lions Lemma, embeds compactly in $L^2(Q_T)$ and via interpolation also in $L^2(0,T; H^1(\Omega))$. This means that $\mathcal{T}$ is also compact. Standard arguments yield the continuity of $\mathcal{T}$ in a straightforward way. Finally, a $\sigma-$uniform bound on the elements of the set
\begin{align*}
\{ \rho\in L^2(0,T; H^{1}(\Omega))\quad\vert\quad \mathcal T(\rho,\sigma)=\rho \}
\end{align*}
follows immediately from \eqref{fixp.2} by choosing $u=\rho$ and applying an interpolation inequality as well as Gronwall's lemma. Therefore Leray-Schauder fix point theorem yields the existence of a fix point $\rho=\rhoed$ for $\mathcal{T}(\cdot,1)$, that is, a solution $\rhoed\in L^2(0,T; H^{m}(\Omega))$ to \eqref{app.1}--\eqref{app.1.ic}.\medskip\\
{\bf Uniform estimates and compactness argument.}
Now we aim at taking the limit $\eps\to 0$ in \eqref{app.1}--\eqref{app.1.ic}. First notice that testing \eqref{app.1} against $\rhoed\in L^2(0,T; H^m(\Omega))$ and exploiting the non-negativity of $f_\delta'$ lead to the bounds
\begin{align}\label{app.est}
\|\rhoed\|_{L^\infty(0,T; L^2(\Omega))} +
\delta^{1/2}\|\rhoed\|_{L^2(0,T; H^{1}(\Omega))}
+ \eps^{1/2}\|\rhoed\|_{L^2(0,T; H^m(\Omega))}\leq 
C\|\rho_0\|_2 .
\end{align}
A uniform bound for the time derivative $\pa_t\rhoed$ in $L^2(0,T; H^m(\Omega)')$ is also derived in a straightforward way from \eqref{app.1} and \eqref{app.est}. From Aubin-Lions Lemma it follows that (up to subsequences) $\rhoed$ is strongly convergent as $\eps\to 0$ in $L^2(Q_T)$ towards $\rhod\in L^2(0,T; H^1(\Omega))$, which is the solution to
\begin{align}\label{app.2}
&\int_0^T\langle \pa_t\rhod , \phi\rangle dt 
+ \delta\int_0^T (\rhod,\phi)_{H^1} dt = -\int_0^T\int_{\Omega}\nabla f_\delta(\rhod)\cdot\nabla\phi dx dt\\ 
&\nonumber
\qquad
\forall\phi\in L^2(0,T; H^1(\Omega)),\\
\label{app.2.ic}
&\rhod(0) = \rho_0\quad\mbox{in }\Omega.
\end{align}
The next step is taking the limit $\delta\to 0$ in \eqref{app.2}--\eqref{app.2.ic}. Choosing $\phi=f_\delta(\rhod)\in L^2(0,T; H^1(\Omega))$ in \eqref{app.2} leads to
\begin{align*}
\int_\Omega F_\delta(\rhod(T))dx 
&+ \delta\int_0^T\int_\Omega\rhod f_\delta(\rhod)dx d\tau 
+ \delta\int_0^T\int_\Omega |\nabla\rhod|^2 f_\delta'(\rhod)dx d\tau \\ 
&+ \int_0^T\int_\Omega |\nabla f_\delta(\rhod)|^2 dx d\tau
=\int_\Omega F_\delta(\rho_0)dx,\qquad t>0,
\end{align*}
where the approximated free energy is given by
\begin{align}\label{F.delta}
F_\delta(\rho)\equiv\int_0^\rho f_\delta(u)du ,\qquad\rho>0.
\end{align}
This yields the following bound
\begin{align}\label{app.est.2}
\|\nabla f_\delta(\rhod)\|_{L^2(Q_T)}^{2} \leq 
\int_\Omega F_\delta(\rho_0)dx\leq C(\|\rho_0\|_\infty),
\end{align}
where the last inequality holds since $0<\delta<\|\rho_{0}\|_\infty^{-1}$.
Furthermore, the weak lower semicontinuity of the $L^2(0,T; H^1(\Omega))$ norm and \eqref{app.est} yield
\begin{align}\label{app.2.est}
\|\rhod\|_{L^\infty(0,T; L^2(\Omega))} +
\delta^{1/2}\|\rhod\|_{L^2(0,T; H^{1}(\Omega))}\leq 
C\|\rho_0\|_2 .
\end{align}

A $\delta-$uniform $L^\infty(Q)$ bound for $\rhod$ can be proved via Stampacchia truncation. Let $M\equiv \sup_\Omega \rho_0$.
Choosing $\phi = (\rhod - M)_+\in L^2(0,T; H^1(\Omega))$ in \eqref{app.2} and exploiting the non-negativity of $f'$ yield
\begin{align*}
\int_\Omega (\rhod(t) - M)_+^2 dx\leq \int_\Omega (\rho_0 - M)_+^2 dx = 0
\end{align*}
meaning that $\rhod(t) \leq M = \|\rho_0\|_\infty$ a.e.~in $\Omega$, $t>0$.\medskip\\
This means that the vector field 
\begin{align*}
U^{(\delta)} = (f_\delta(\rhod),0,\ldots,0)
\end{align*}
satisfies
\begin{align*}
\|U^{(\delta)}\|_{L^{\infty}(Q_T)}\leq C, \quad
\mbox{ $\mbox{Curl}_{(t,x)}U^{(\delta)}$ is relatively compact in $W^{-1,r}(Q_T)$ for some $r>1$ }
\end{align*}
since 
\begin{align*}
\|\mbox{Curl}_{(t,x)}U^{(\delta)}\|_{L^2(Q_T)}\leq C \|\nabla f_\delta(\rhod)\|_{L^2(Q_T)}.
\end{align*}
Furthermore, the vector field
\begin{align*}
V^{(\delta)} = (\rhod , -\nabla f_\delta(\rhod) )
\end{align*}
satisfies
\begin{align*}
\|V^{(\delta)}\|_{L^2(Q_T)}\leq C, \quad
\mbox{ $\mbox{div}_{(t,x)}V^{(\delta)}$ is relatively compact in $W^{-1,r}(Q_T)$ for some $r>1$ }
\end{align*}
since 
\begin{align*}
\Div_{(t,x)}V^{(\delta)} = -\delta\rhod + \delta\Delta\rhod \to 0\quad\mbox{strongly in }L^2(0,T; H^1(\Omega)').
\end{align*}
Therefore, the Div-Curl Lemma allows us to deduce (the bar denotes weak limit)
\begin{align*}
\overline{U^{(\delta)}\cdot V^{(\delta)} } = \overline{U^{(\delta)}}\cdot\overline{V^{(\delta)}}
\end{align*}
and so
\begin{align*}
\overline{f_\delta(\rhod) \rhod } = \overline{f_\delta(\rhod)}\,\overline{\rhod}\quad\mbox{a.e.~in }Q_T.
\end{align*}
On the other hand, $f_\delta(r)=f(r)$ for $r\leq\delta^{-1}$,
which means, given the uniform $L^\infty$ bound for $\rhod$, that for $\delta>0$ small enough $f_\delta(\rhod)=f(\rhod)$ a.e.~in $Q_T$. Therefore
\begin{align*}
\overline{f(\rhod) \rhod } = \overline{f(\rhod)}\,\overline{\rhod}\quad\mbox{a.e.~in }Q_T .
\end{align*}
Being $f$ monotone, we deduce
\begin{align*}
\overline{f(\rhod)} = f(\overline{\rhod})\quad\mbox{a.e.~in }Q_T.
\end{align*}
Let $\rho = \overline{\rhod}$. We just proved that $f_\delta(\rhod)\rightharpoonup^* f(\rho)$ weakly* in $L^\infty(Q_T)$ as $\delta\to 0$ (up to subsequences). Furthermore \eqref{app.est.2} implies the $L^2(Q_T)-$weak convergence of $\nabla f_\delta(\rhod)$, meaning that
\begin{align*}
\nabla f_\delta(\rhod)\rightharpoonup\nabla f(\rho)\quad\mbox{weakly in }L^2(Q_T).
\end{align*}
A uniform bound for $\pa_t\rhod$ in $L^2(0,T; H^1(\Omega)')$ can be easily derived from \eqref{app.2} and \eqref{app.est.2} implying that 
\begin{align*}
\pa_t\rhod \rightharpoonup\pa_t\rho\quad\mbox{weakly in }L^2(0,T; H^1(\Omega)').
\end{align*}
We are therefore able to take the limit $\delta\to 0$ in \eqref{app.2} and obtain a weak solution to \eqref{1}--\eqref{2}.\medskip\\
\begin{align*}
	\rho_1,\, \rho_2\in L^\infty(Q),\quad \nabla f(\rho_1),\,\na f(\rho_2)\in L^2(Q),\quad \pa_t\rho_1,\, \pa_t\rho_2\in L^2(0,\infty; H^{1}(\Omega)'),
\end{align*}
with the same initial datum: $\rho_1(\cdot,0)=\rho_2(\cdot,0)=\rho_0$ in $\Omega$. Therefore, it holds
\begin{equation}
	\label{rho12.eq}
	\begin{cases}
	\pa_t(\rho_1 - \rho_2) = \Delta(f(\rho_1)-f(\rho_2))&\mbox{in }Q,\\
	\nu\cdot\nabla (f(\rho_1)-f(\rho_2)) = 0&\mbox{on }\pa\Omega\times (0,T),\\
   	(\rho_1 - \rho_2)(\cdot,0) = 0 &\mbox{on }\Omega.
   	\end{cases}
\end{equation}
Define the function $u : Q\to\R$ as the unique solution to
\begin{equation}
	\label{def.u}
	\begin{cases}
		-\Delta u = \rho_1 - \rho_2 &\quad\mbox{in }\Omega,~~t>0,\\
		\nu\cdot\nabla u=0 & \quad\mbox{on }\pa\Omega,~~t>0,\\
		\int_\Omega u dx = 0& \quad t>0.
	\end{cases}
\end{equation}
Clearly $u\in L^\infty(0,T; H^1(\Omega))$. Thus, we can use $u$ as test function in \eqref{rho12.eq}. We obtain
\begin{align*}
	\int_\Omega|\na u(t)|^2 dx + \int_0^t\int_\Omega(\rho_1 - \rho_2)(f(\rho_1)-f(\rho_2))dx d\tau = 0,\quad t\in (0,T).
\end{align*}
Being $f$ non-decreasing, it follows that $u\equiv 0$ in $Q$ and so $\rho_1\equiv\rho_2$ on $Q$. The weak solution is therefore unique.\medskip\\
{\bf Monotonicity property.}
We now prove that, for a.e.~$x\in\Omega$, the mapping $t\in (0,\infty)\mapsto \min(\rho(x,t),\rho_{cr})\in\R$ is non-decreasing.

Let $\psi\in C^\infty_c(\Omega)$, $\psi\geq 0$ a.e.~in $\Omega$,
{$0<t_0<t$} arbitrary.
Since $\rhod\in L^2(0,T; H^1(\Omega))\cap L^\infty(Q_T)$, we can test \eqref{app.2} against $(\rhod - \rho_{cr})_-^3\psi$ and obtain
\begin{align*}
\frac{1}{4}&\int_\Omega (\rhod(t) - \rho_{cr})_-^4 \psi dx 
+ \delta\int_{t_0}^t\int_\Omega\rhod (\rhod - \rho_{cr})_-^3\psi dxdt'
+ 3\delta\int_{{t_0}}^t\int_\Omega (\rho - \rho_{cr})_-^2 |\na\rhod|^2\psi dxdt'\\
& +\delta\int_{{t_0}}^t\int_\Omega (\rhod - \rho_{cr})_-^3\na\psi\cdot\na\rhod dx
= \frac{1}{4}\int_\Omega (\rhod({t_0}) - \rho_{cr})_-^4\psi dx\\
& -\int_{{t_0}}^t\int_\Omega f_\delta'(\rhod)(\rhod - \rho_{cr})_-^3\na\rhod\cdot\na\psi dx dt'
- 3\int_{{t_0}}^t\int_\Omega f_\delta'(\rhod)(\rhod - \rho_{cr})_-^2 |\na\rhod|^2\psi dxdt'.
\end{align*}
However, since $f_\delta(s) = f(s) = 0$ for $s\leq\rho_{cr}$, the last two integrals on the right-hand side of the above identity vanish. Given that the third integral on the left-hand side is non-negative (since $\psi\geq 0$), we get
\begin{align*}
\frac{1}{4}&\int_\Omega (\rhod(t) - \rho_{cr})_-^4 \psi dx 
+ \delta\int_{{t_0}}^t\int_\Omega\rhod (\rhod - \rho_{cr})_-^3\psi dxdt'\\
& +\delta\int_{{t_0}}^t\int_\Omega (\rhod - \rho_{cr})_-^3\na\psi\cdot\na\rhod dx
\leq \frac{1}{4}\int_\Omega (\rhod({t_0}) - \rho_{cr})_-^4\psi dx.
\end{align*}
An integration by parts in the third integral on the left-hand side of the above inequality yields
\begin{align*}
\frac{1}{4}&\int_\Omega (\rhod(t) - \rho_{cr})_-^4 \psi dx 
+\delta\int_{{t_0}}^t\int_\Omega\left(
\rhod (\rhod - \rho_{cr})_-^3\psi - \frac{1}{4}(\rhod - \rho_{cr})_-^4\Delta\psi \right)dx dt'\\
&\leq \frac{1}{4}\int_\Omega (\rhod({t_0}) - \rho_{cr})_-^4\psi dx.
\end{align*}
Given the choice of $\psi$ and the uniform $L^\infty(Q_T)$ bounds for $\rhod$, the second integral on the left-hand side of the above inequality vanishes in the limit $\delta\to 0$, yielding
\begin{align*}
\int_\Omega (\rho(t) - \rho_{cr})_-^4 \psi dx \leq 
\int_\Omega (\rho({t_0}) - \rho_{cr})_-^4\psi dx,\quad t>{t_0>}0.
\end{align*}
Since $\psi$ is an arbitrary non-negative test function, we infer
\begin{align*}
(\rho(t) - \rho_{cr})_-^4 \leq (\rho({t_0}) - \rho_{cr})_-^4\quad
\mbox{a.e.~in }\Omega,~~ t>{t_0>}0.
\end{align*}
Since $(x)_- = -(-x)_+$ for every $x\in\R$, it follows
\begin{align*}
(\rho_{cr} - \rho(t))_+ \leq (\rho_{cr} - \rho({t_0}))_+\quad
\mbox{a.e.~in }\Omega,~~ t>{t_0>}0.
\end{align*}
Given that $\min(s,\rho_{cr}) = \rho_{cr} + (s - \rho_{cr})_-
= \rho_{cr} - (\rho_{cr}-s)_+$ for $s\in\R$, we conclude that, for a.e.~$x\in\Omega$, the mapping $t\in (0,\infty)\mapsto \min(\rho(x,t),\rho_{cr})\in\R$ is non-decreasing.
This finishes the proof of Theorem~\ref{thm.existence}.

\begin{rem} The monotonicity property showed in Theorem~\ref{thm.existence} implies in particular that
\begin{align*}
\inf_{\Omega\times(0,\infty)}\rho \geq \inf_\Omega\rho_0.
\end{align*}
As a consequence, if $\rho_0$ is uniformly positive in $\Omega$, then $\rho$ is uniformly positive in $\Omega\times(0,\infty)$.
\end{rem}

\section{Long-time behavior (proof of theorem \ref{thm.ltb.diff})}\label{sec.longtime}
In this section, we concentrate on proving Theorem \ref{thm.ltb.diff}. 
The proof is divided into several lemmas.

The following lemma yields the exponential convergent rate for the solution in the supercritical case.
\begin{lemma}\label{lem.exp}
	Let $F(\rho\vert\rho_\infty)$ as in \eqref{rel.entr}.
	If $\rho_\infty>\rho_{cr}$, then a constant $\lambda>0$ exists such that
	\begin{align*}
	\int_\Omega F(\rho(t)\vert\rho_\infty)dx \leq e^{-\lambda t}\int_\Omega F(\rho_0\vert\rho_\infty)dx\quad t>0.
	\end{align*}
\end{lemma}
\begin{proof}
	Testing \eqref{1} against $f(\rho)-f(\rho_\infty)$ yields
	\begin{align}\label{sunny.1}
	\frac{d}{dt}\int_\Omega F(\rho(t)\vert\rho_\infty)dx = - \int_\Omega |\na f(\rho(t))|^2 dx.
	\end{align}
	Lemma \ref{lem.new.f} (see Appendix) implies
	\begin{align*}
	\frac{d}{dt}\int_\Omega F(\rho(t)\vert\rho_\infty)dx \leq 
	- c\int_\Omega |f(\rho(t)) - f(\rho_\infty)|^2 dx.
	\end{align*}
	On the other hand, since $\rho_\infty > \rho_{cr}$, it follows
	\begin{align*}
	F(\rho(t)\vert\rho_\infty) \leq C |f(\rho(t)) - f(\rho_\infty)|^2,\quad t>0.
	\end{align*}
	Indeed, $\rho\in L^\infty(\Omega\times (0,\infty))$, so the above inequality will follow provided that
	\begin{align*}
	\frac{F(s\vert\rho_\infty)}{|f(s) - f(\rho_\infty)|^2}\leq C\quad\mbox{as }s\to\rho_{\infty},
	\end{align*}
	which is easily verified by employing De L'Hospital's theorem. We deduce
	\begin{align*}
	\frac{d}{dt}\int_\Omega F(\rho(t)\vert\rho_\infty)dx \leq 
	- \lambda\int_\Omega F(\rho(t)\vert\rho_\infty) dx,\quad t>0.
	\end{align*}
	Gronwall's Lemma yields the statement. This finishes the proof of the Lemma.
\end{proof}

\begin{lemma}\label{lem.lt.easy}
Let $\rho$ be the weak solution to \eqref{1}--\eqref{2} according to Theorem~\ref{thm.existence}.\smallskip\\
If $\rho_\infty \geq \rho_{cr}$, then 
\begin{align}\label{cnv.super.lemma}
\rho(t)\to\rho_\infty \quad \mbox{ as } t\to\infty \mbox{ strongly in } L^p(\Omega),\quad  \forall 1\leq p < \infty.
\end{align}
If $\rho_\infty < \rho_{cr}$, then
\begin{align}\label{cnv.sub.lemma}
(\rho(t)-\rho_{cr})_+\to 0 \quad \mbox{ as } t\to\infty \mbox{ strongly in } L^p(\Omega),\quad  \forall 1\leq p < \infty.
\end{align}
\end{lemma}

\begin{proof}
We distinguish three cases:\medskip\\
{\em Case 1: $\rho_\infty > \rho_{cr}$.}
Relation \eqref{cnv.super.lemma} follows immediately from Lemma \ref{lem.exp} and the $L^\infty(\Omega\times (0,\infty))$ bounds for $\rho$.
\medskip\\
{\em Case 2: $\rho_\infty < \rho_{cr}$.} 
Integrating \eqref{sunny.1} in time yields
\begin{align*}
\int_{\Omega}F(\rho(t))dx + \int_0^t\int_\Omega |\nabla f(\rho)|^2 dx d\tau = \int_{\Omega}F(\rho_0)dx,\quad t>0.
\end{align*}
Due to the above free energy identity, we have a sequence $t_n\to\infty$ such that 
\begin{align*}
\int_\Omega |\nabla f(\rho(t_n))|^2 dx \to 0\quad\mbox{as }n\to\infty .
\end{align*}
Poincar\'e's Lemma yields
\begin{align*}
f(\rho(t_n)) - \avg{f(\rho(t_n))}\to 0\quad\mbox{strongly in }L^2(\Omega).
\end{align*}
On the other hand, the time-uniform $L^\infty$ bounds for $\rho$ imply that
$\avg{f(\rho(t_n))}$ is bounded, from which we obtain a subsequence $(t_{n_k})$ of $(t_n)$ such that 
\begin{align*}
\avg{f(\rho(t_{n_k}))}\to \zeta\quad\mbox{as }k\to\infty .
\end{align*}
Therefore,
\begin{align}\label{frho.tnk}
f(\rho(t_{n_k})) \to \zeta \geq 0\quad\mbox{strongly in }L^2(\Omega).
\end{align}
We claim that the number $\zeta$ appearing in \eqref{frho.tnk} is zero. In fact, assume by contradiction that $\zeta>0$. 
From \eqref{hp.f} it follows that $\rho(t_{n_k})\to f^{-1}(\zeta)>\rho_{cr}$ a.e.~in $\Omega$, implying thanks to the $L^\infty(\Omega)$ bounds for $\rho(t_{n_k})$ that 
$\rho(t_{n_k})\to f^{-1}(\zeta)>\rho_{cr}$ strongly in $L^1(\Omega)$, against mass conservation. So $\zeta = 0$, meaning that $f(\rho(t_{n_k}))\to 0$ a.e.~in $\Omega$. 
Since \eqref{hp.f} holds, it follows that 
\begin{align*}
\limsup_{k\to\infty}\rho(t_{n_k})\leq\rho_{cr}\quad\mbox{a.e.~in }\Omega.
\end{align*}
From the definition of $F$ and \eqref{hp.f} we deduce
\begin{align*}
\limsup_{k\to\infty}F(\rho(t_{n_k}))\leq 
F(\limsup_{k\to\infty}\rho(t_{n_k}))\leq
F(\rho_{cr})=0 \quad\mbox{a.e.~in }\Omega,
\end{align*}
which implies (given that $F$ is non-negative)
\begin{align*}
F(\rho(t_{n_k}))\to 0\quad\mbox{a.e.~in }\Omega \,.
\end{align*}
Given the uniform $L^\infty$ bounds for $\rho$, it follows that 
\begin{align*}
\lim_{k\to\infty}\int_\Omega F(\rho(t_{n_k})) dx = 0.
\end{align*}
However, since
\begin{align*}
\frac{d}{dt}\int_\Omega F(\rho(t)) dx = -\int_\Omega |\nabla f(\rho(t))|^2 dx\leq 0,
\end{align*}
the function $t\mapsto \int_\Omega F(\rho(t)) dx$ is non-increasing, meaning that
\begin{align*}
\lim_{t\to\infty}\int_\Omega F(\rho(t)) dx = 0.
\end{align*}
It follows that $F(\rho(t))\to 0$ a.e.~in $\Omega$, easily implying that 
\begin{align*}
\limsup_{t\to\infty}\rho(t)\leq\rho_{cr}\quad\mbox{a.e.~in }\Omega.
\end{align*}
This in turn yields that $(\rho(t)-\rho_{cr})_+\to 0$ a.e.~in $\Omega$. This fact, together with the $L^\infty$ bounds for $\rho$, yields 
\eqref{cnv.sub.lemma}. \medskip\\
{\em Case 3: $\rho_\infty = \rho_{cr}$.} It is straightforward to see that the arguments in Case 2 apply also to this case, assuming that \eqref{cnv.sub.lemma} holds. However, mass conservation implies
\begin{align*}
\int_\Omega |(\rho(t)-\rho_{cr})_-| dx = &
-\int_\Omega (\rho(t)-\rho_{\infty})_- dx = \int_\Omega (\rho(t)-\rho_{\infty})_+ dx - \int_\Omega (\rho(t)-\rho_{\infty}) dx\\
= & \int_\Omega (\rho(t)-\rho_{cr})_+ dx\to 0\quad\mbox{as }t\to\infty,
\end{align*}
meaning that $\rho(t)-\rho_{cr}\to 0$ strongly in $L^1(\Omega)$ as $t\to\infty$. This fact, together with the $L^\infty$ bounds for $\rho$, yields \eqref{cnv.super}.
This finishes the proof of the lemma.
\end{proof}

The results in Theorem~\ref{thm.ltb.diff} related to the supercritical case $\rho_\infty \geq \rho_{cr}$ have now been proven. We proceed with the subcritical case $\rho_\infty < \rho_{cr}$.
\begin{lemma}\label{lem.new}
	Let $\rho$ be the weak solution to \eqref{1} according to Theorem~\ref{thm.existence}, and let $\rho_\infty < \rho_{cr}$.
	The following statements hold:
	\begin{enumerate}
		\item A function $\hat{\rho}\in L^\infty(\Omega)$ exists such that 
		$\rho(t)\to\hat\rho$ strongly in $L^p(\Omega)$ for every $p<\infty$ as $t\to\infty$. Furthermore, it holds $\hat{\rho}\leq\rho_{cr}$ a.e.~in $\Omega$.
		\item If $\mbox{meas}(\{ \rho_0 > \rho_{cr} \}) > 0$, then
		there exists $T^*>0$ such that 
		$\int_0^{T^*}\int_\Omega |\nabla\rho|^2 dx dt = \infty$.
	\end{enumerate}
\end{lemma}

\begin{proof}
	Let us show that the statement 1 is true. From Theorem~\ref{thm.existence} it follows that $\min(\rho(t),\rho_{cr})$ is a.e.~convergent in $\Omega$ as $t\to\infty$. Furthermore, from Lemma \ref{lem.lt.easy} it follows that $(\rho(t)-\rho_{cr})_+$ is a.e.~convergent in $\Omega$ as $t\to\infty$. Since $\rho(t) = \min(\rho(t),\rho_{cr}) + (\rho(t)-\rho_{cr})_+$, we deduce that $\rho(t)$ is a.e.~convergent in $\Omega$ as $t\to\infty$ towards some limit function $\hat\rho$.
	However, given the uniform $L^\infty$ bounds for $\rho$, we conclude
	that $\rho(t)\to\hat\rho$ strongly in $L^p(\Omega)$ for every $p<\infty$. The upper bound on $\hat{\rho}$ follows immediately from Theorem~\ref{thm.ltb.diff}. Thus Statement 1 holds.
	
	Let us now prove the second statement. We proceed by contradiction. Assume that $\mbox{meas}(\{ \rho_0 > \rho_{cr} \}) > 0$ and that
	$\na\rho\in L^2(0,T; L^2(\Omega))$ for every $T>0$.
	It follows that, given any $\psi\in C^\infty_c(\Omega)$, we can employ $(\rho - \rho_{cr})_-^3\psi$ as a test function in \eqref{1},
	obtaining
	\begin{align*}
	\frac{1}{4}&\int_\Omega (\rho(t) - \rho_{cr})_-^4 \psi dx 
	= \frac{1}{4}\int_\Omega (\rho_0 - \rho_{cr})_-^4\psi dx\\
	& -\int_0^t\int_\Omega (\rho - \rho_{cr})_-^3
	\na f(\rho)\cdot\na\psi dx dt'
	- \int_0^t\int_\Omega \na[(\rho - \rho_{cr})_-^3]\cdot 
	\na f(\rho)\, \psi dxdt'
	\end{align*}
	for $t\in (0,T)$.
	Since $\na\rho\in L^2(0,T; L^2(\Omega))$ we can rewrite the above identity as
	\begin{align*}
	\frac{1}{4}&\int_\Omega (\rho(t) - \rho_{cr})_-^4 \psi dx 
	= \frac{1}{4}\int_\Omega (\rho_0 - \rho_{cr})_-^4\psi dx\\
	& -\int_0^t\int_\Omega (\rho - \rho_{cr})_-^3 f'(\rho)
	\na \rho\cdot\na\psi dx dt'
	- 3\int_0^t\int_\Omega (\rho - \rho_{cr})_-^2 f'(\rho)
	|\na \rho|^2\, \psi dxdt'.
	\end{align*}
	Clearly the last two integrals on the right-hand side of the above inequality vanish due to the properties of $f$. We deduce
	\begin{align*}
	\frac{1}{4}&\int_\Omega (\rho(t) - \rho_{cr})_-^4 \psi dx 
	= \frac{1}{4}\int_\Omega (\rho_0 - \rho_{cr})_-^4\psi dx
	\end{align*}
	for every $\psi\in C^\infty_c(\Omega)$, yielding
	$(\rho(t) - \rho_{cr})_-^4 = (\rho_0 - \rho_{cr})_-^4$ a.e.~in $\Omega$, and so
	\begin{align*}
	\min(\rho(t),\rho_{cr}) = \rho_{cr} +
	(\rho(t) - \rho_{cr})_- = (\rho_0 - \rho_{cr})_- + \rho_{cr}
	= \min(\rho_0,\rho_{cr})\quad\mbox{a.e.~in }\Omega,~~ t\in (0,T).
	\end{align*}
	Being $T$ arbitrary, it means that the above identity holds for every $t>0$. As a consequence,
	\begin{align*}
	\rho(t) = \min(\rho(t),\rho_{cr}) + (\rho(t)-\rho_{cr})_+ = 
	\min(\rho_0,\rho_{cr}) + (\rho(t)-\rho_{cr})_+,\quad t>0.
	\end{align*}
	Taking the limit $t\to\infty$ on both sides of the above equality and employing Lemma \ref{lem.lt.easy} as well as Statement 1 lead to
	\begin{align*}
	\hat\rho = \min(\rho_0,\rho_{cr})\quad\mbox{a.e.~in }\Omega\,.
	\end{align*}
	However, the above identity together with the assumption on the initial datum yield
	\begin{align*}
	\int_\Omega (\rho_0 - \hat\rho)dx = 
	\int_\Omega(\rho_0 - \min(\rho_0,\rho_{cr}))dx = 
	\int_{ \{\rho_0 > \rho_{cr}\} }(\rho_0 - \rho_{cr})dx > 0,
	\end{align*}
	against mass conservation. Therefore, Statement 2 holds.
	This finishes the proof of the lemma.
\end{proof}

\begin{lemma}\label{lem.new.3}
	Let $\Omega$ connected, $F$ as in \eqref{rel.entr}, $\rho_\infty < \rho_{cr}$ and assume \eqref{hp.kappa} holds. Then,
	\begin{align}\label{F.decay.lem}
	\int_\Omega F(\rho(t))dx \leq C t^{- \frac{\kappa + 1}{\kappa - 1} }\quad\mbox{as }t\to\infty.
	\end{align}
\end{lemma}

\begin{proof}
	Putting \eqref{sunny.1} and Lemma \ref{lem.new.f} in the Appendix together yields
	\begin{align*}
	\frac{d}{dt}\int_\Omega F(\rho(t))dx \leq -c\int_\Omega |f(\rho(t))|^2 dx.
	\end{align*}
	From assumption \eqref{hp.kappa} it follows
	\begin{align*}
	F(s)\leq C_R\int_{\rho_{cr}}^s (r-\rho_{cr})^\kappa dr = \frac{C_R}{1+\kappa}(s-\rho_{cr})^{\kappa + 1}
	\leq \frac{C_R c_R^{-(\kappa + 1)/\kappa} }{1+\kappa}f(s)^{1 + 1/\kappa}
	\quad s\leq R,
	\end{align*}
	and given the $L^\infty(\Omega\times(0,\infty))$ bounds for $\rho$ we obtain
	\begin{align*}
	\frac{d}{dt}\int_\Omega F(\rho(t))dx \leq -c\int_\Omega F(\rho(t))^{ \frac{2\kappa}{\kappa + 1}} dx \leq 
	-c\left(\int_\Omega F(\rho(t)) dx\right)^{ \frac{2\kappa}{\kappa + 1}},
	\end{align*}
	where the last inequality follows from Jensen's inequality. Gronwall's Lemma yields the statement. This finishes the proof of the Lemma.
\end{proof}

\begin{lemma}\label{lem.new.2}
	Assume that $\Omega$ is connected, $\rho_\infty < \rho_{cr}$ and \eqref{hp.kappa} holds.
	Then,
	\begin{align*}
	\|(\rho(t)-\rho_{cr})_+\|_{L^1(\Omega)} \leq C\, t^{-\frac{1}{\kappa-1}},\qquad 
	\|\rho(t)-\hat{\rho}\|_{\dot{H}^1(\Omega)'}\leq C\, t^{-\frac{1}{\kappa - 1}},
	\end{align*}
	for $t\to\infty$.
\end{lemma}

\begin{proof}[Proof.]
	Lemma \ref{lem.new.3} and Assumption \ref{hp.kappa} lead to 
	\begin{align*}
	\int_\Omega (\rho-\rho_{cr})_+^{\kappa + 1} dx \leq C
	\int_\Omega F(\rho)dx \leq C t^{- \frac{\kappa + 1}{\kappa - 1}}\quad\mbox{as }t\to\infty .
	\end{align*}
	Jensen's inequality yields
	\begin{align*}
	\int_\Omega (\rho-\rho_{cr})_+ dx \leq C t^{- \frac{1}{\kappa - 1}}\quad\mbox{as }t\to\infty .
	\end{align*}
	We now prove the statement related to the negative homogeneous Sobolev seminorm. 
	Define for $t\geq 0$
	\begin{align*}
	\begin{cases}
	-\Delta V(t) = \rho(t) - \langle\rho(t)\rangle & \mbox{in }\Omega\\
	\nu\cdot\nabla V(t) = 0 & \mbox{on }\partial\Omega\\
	\int_\Omega V(t) dx = 0 & 
	\end{cases},\qquad
	\begin{cases}
	-\Delta \hat{V} = \hat{\rho} - \langle\hat{\rho}\rangle & \mbox{in }\Omega\\
	\nu\cdot\nabla \hat{V} = 0 & \mbox{on }\partial\Omega\\
	\int_\Omega \hat{V} dx = 0 & 
	\end{cases}.
	\end{align*}
	Since $\langle\rho(t)\rangle = \langle\hat{\rho}\rangle$, we deduce
	\begin{align*}
	\|\rho(t) - \hat{\rho}\|_{\dot{H}^{1}(\Omega)'} = 
	\| \nabla( V(t) - \hat{V} ) \|_{L^2(\Omega)}.
	\end{align*}
	Let us compute
	\begin{align*}
	\frac{d}{dt} \frac{1}{2}\| \nabla( V(t) - \hat{V} ) \|_{L^2(\Omega)}^2 &= 
	\int_\Omega \nabla( V(t) - \hat{V} )\cdot \nabla\partial_t V(t)\, dx\\
	&= \int_\Omega ( V(t) - \hat{V} )\pa_t\rho(t)\, dx\\
	&= \int_\Omega ( V(t) - \hat{V} )\Delta f(\rho(t))\, dx .
	\end{align*}
	Integrating by parts twice leads to
	\begin{align*}
	-\frac{d}{dt} \frac{1}{2}\| \nabla( V(t) - \hat{V} ) \|_{L^2(\Omega)}^2 &= 
	\int_\Omega (\rho(t)-\hat{\rho})f(\rho(t))\, dx.
	\end{align*}
	The uniform $L^\infty$ bounds for $\rho$ and Assumption \eqref{hp.kappa} imply
	\begin{align*}
	-\frac{d}{dt} \frac{1}{2}\| \nabla( V(t) - \hat{V} ) \|_{L^2(\Omega)}^2 &\leq 
	C\int_\Omega f(\rho(t))\, dx\\
	&\leq C\int_\Omega (\rho(t) - \rho_{cr})_+^{\kappa} dx\\
	&\leq C\int_\Omega F(\rho(t))^{\frac{\kappa}{\kappa + 1}}dx .
	\end{align*}
	Jensen's inequality and Lemma \ref{lem.new.3} yield
	\begin{align*}
	-\frac{d}{dt} \frac{1}{2}\| \nabla( V(t) - \hat{V} ) \|_{L^2(\Omega)}^2 &\leq 
	C\left(\int_\Omega F(\rho(t)) dx\right)^{\frac{\kappa}{\kappa + 1}}\\
	&\leq C t^{- \frac{\kappa}{\kappa - 1} }\quad\mbox{as }t\to\infty.
	\end{align*}
	Integrating the above inequality between $t$ and $\infty$ and noticing that 
	$\lim_{t\to\infty}\| \nabla( V(t) - \hat{V} ) \|_{L^2(\Omega)} = 0$ since $\rho(t)\to \hat{\rho}$ (thanks to Lemma \ref{lem.new}), we obtain
	\begin{align*}
	\| \nabla( V(t) - \hat{V} ) \|_{L^2(\Omega)}^2 \leq C t^{- \frac{1}{\kappa - 1} }\quad\mbox{as }t\to\infty.
	\end{align*}
	This finishes the proof of the lemma.
\end{proof}

This finishes the proof of Theorem \ref{thm.ltb.diff}.

\begin{rem}
We point out that for $\rho_\infty > \rho_{cr}$ it holds $F(s\vert\rho_{\infty})\geq c |s-\rho_{\infty}|^2$ for $s\geq 0$, so \eqref{lt.exp} implies that $\rho(t)\to\rho_{\infty}$ in $L^2(\Omega)$ with an exponential rate.
\end{rem}

\section{Proof of Theorem \ref{thm.fb}}\label{sec.fb}

\subsection{Existence for the free boundary problem \eqref{fb}--\eqref{fb.ic}}
We show in this section that the free boundary problem \eqref{fb}--\eqref{fb.ic} has a weak solution, that is, we prove the first part of Theorem~\ref{thm.fb}.
We start with a reformulation via a change of spatial variables: $x = R(t)y$. Let $v(y,t)=u(x,t)=u(R(t)y,t)$ for $y\in B\equiv B(0,1)$, $t>0$. Then,
\begin{align*}
\pa_t v(y,t) = R'(t)y\cdot\nabla_x u(R(t)y,t) + \pa_t u(R(t)y,t),
\quad \nabla_x u(R(t)y,t) = R(t)^{-1}\nabla_y v(y,t),
\end{align*}
and we deduce
\begin{align*}
R(t)^2\pa_t v(y,t) = R(t)R'(t)y\cdot\nabla_y v(y,t) + \Delta_y f(v(y,t)).
\end{align*}
On the other hand, for $x\in\pa E(t)$ (i.e.~$y\in\pa B$)
\begin{align*}
R'(t)= \frac{|\nabla_x f(u(x,t))|}{\rho_{cr}-\rho_0(x)} = 
\frac{|\nabla_y f(v(y,t))|}{R(t)(\rho_{cr}-\rho_0(R(t)y))}.
\end{align*}
Being the right-hand side of the above identity a radial function in $B$, we can replace it with its spatial average
on $\pa B$:
\begin{align*}
R'(t) = \fint_{\pa B}\frac{|\nabla_y f(v(y,t))|}{R(t)(\rho_{cr}-\rho_0(R(t)y))} d\sigma_y.
\end{align*}
For every Lebesgue-measurable function $\omega : B\times [0,\infty)\to\R$ such that $\na_y\omega(t)\in L^1(\pa B)$ for $t\geq 0$ we define 
\begin{align}\label{g.map}
g[\omega](r,t) = \fint_{\pa B}\frac{|\nabla_y \omega(y,t)|}{\rho_{cr}-\rho_0(r y)} d\sigma_y,\quad r>R_0,\quad t\geq 0.
\end{align}
Thus, we obtain
\begin{align}\label{free.v}
R(t)^2\pa_t v(y,t) &= g[f(v)](R(t),t)y\cdot\nabla_y v(y,t) + \Delta_y f(v(y,t))\quad (y,t)\in B\times (0,T),\\
\label{free.v.bc}
v(y,t) &= \rho_{cr}\quad (y,t)\in\pa B\times (0,T),\\
\label{free.v.ic}
v(y,0) &= \rho_0(R_0 y)\quad y\in B,\\
\label{free.v.R}
R'(t) &= \frac{g[f(v)](R(t),t)}{R(t)}\quad t>0,\qquad R(0)=R_0.
\end{align}
In order to have better estimates (see later part), we define the new variable $w = f(v)$ and $a = f'\circ f^{-1}$, hence obtaining the following {\em semilinear} equation
\begin{align}\label{free.1}
\pa_t w &= R(t)^{-2}g[w](R(t),t)y\cdot\nabla_y w + R(t)^{-2} a(w)\Delta_y w\quad (y,t)\in B\times (0,T),\\
\label{free.1.bc}
w(y,t) &= 0\quad (y,t)\in\pa B\times (0,T),\\
\label{free.1.ic}
w(y,0) &= f(\rho_0(R_0 y))\quad y\in B,\\
\label{free.1.R}
R'(t) &= \frac{g[w](R(t),t)}{R(t)}\quad t>0,\quad R(0) = R_0.
\end{align}
We aim at proving the well-posedness of \eqref{free.1}--\eqref{free.1.R} in a suitable Sobolev space.
The proof is divided into several steps:\medskip\\
{\em Step 1: Regularization.}
Let $\eps>0$, $\delta>0$ arbitrary.
We start by defining the truncation operator $[\cdot]_\eps$ as well as the time-regularizing kernel $\kappa_\delta$
\begin{align}
\label{trunc.def}
&[x]_\eps = \min(1/\eps,\max(\eps,x)),\qquad x\in\R,\\
\label{kappa.def}
&{\kappa}_\delta(t)=\delta^{-1}{\kappa}_1(t\delta^{-1}),\quad 
{\kappa}_1\in C^1_c((0,\infty)),\quad 
{\kappa}_1\geq 0,\quad \int_\R {\kappa}_1 dt = 1,
\end{align}
as well as the regularized mappings $g_\eps$, $g_{\eps,\delta}$
\begin{align}
\label{g.map.eps}
g_{\eps}[\omega](r,t) &= \fint_{\pa B}\frac{
	|\nabla_y \omega(y,t)|}{\rho_{cr} +\eps -\rho_0(r y)} d\sigma_y,\quad r>R_0,\quad t\geq 0,\\
\label{g.map.eps.delta}
g_{\eps,\delta}[\omega](r,t) &= 
\int_0^t \kappa_\delta(t-\tau) g_\eps [\omega](r,\tau) d\tau,\quad r>R_0,\quad t\geq 0.
\end{align}
Then, we consider the regularized problem
\begin{align}\label{free.1.eps}
&\pa_t w_{\eps,\delta} = \frac{g_{\eps,\delta}[w_{\eps,\delta}](R_{\eps,\delta}(t),t)}{R_{\eps,\delta}(t)^2}y\cdot\nabla_y w_{\eps,\delta} + [R_{\eps,\delta}(t)^{-2}]_\eps [a(w_{\eps,\delta})]_\eps\Delta_y w_{\eps,\delta}\quad (y,t)\in B\times (0,T),\\
\label{free.1.bc.eps}
& w_{\eps,\delta}(y,t) = 0\quad (y,t)\in\pa B\times (0,T),\\
\label{free.1.ic.eps}
& w_{\eps,\delta}(y,0) = f(\rho_0(R_0 y))\quad y\in B,\\
\label{free.1.R.eps}
& R_{\eps,\delta}'(t) = \frac{g_{\eps,\delta} [\omega_{\eps,\delta}] (R_{\eps,\delta}(t),t)}{R_{\eps,\delta}(t)}\quad t>0,\quad R(0)=R_0.
\end{align}
We reformulate \eqref{free.1.eps}--\eqref{free.1.R.eps} as a fix point problem for the following operator:
\begin{align*}
S : (\tilde{w},s)\in X\times [0,1] \mapsto w\in X,\quad 
X\equiv L^2(0,T; H^{3/2}(B)),
\end{align*}
such that
\begin{align}
\label{free.fp}
&\pa_t w = 
s \frac{g_{\eps,\delta}[\tilde{w}](R(t),t)}{R(t)^{2}}y\cdot\nabla_y w 
+ (s [R(t)^{-2}]_\eps [a(\tilde{w})]_\eps + 1-s)\Delta_y w\quad (y,t)\in B\times (0,T),\\
\label{free.fp.bc}
&w(y,t) = 0\quad (y,t)\in\pa B\times (0,T),\\
\label{free.fp.ic}
&w(y,0) = f(\rho_0(R_0 y))\quad y\in B,\\
\label{free.fp.R}
&R'(t) = \frac{g_{\eps,\delta}[\tilde{w}](R(t),t)}{R(t)}\quad t>0,\quad R(0)=R_0.
\end{align}
The problem above is uniquely solvable. Precisely, since $\tilde{w}\in X$, it follows that $\nabla\tilde{w}\in L^2(0,T; L^2(\pa B))$, so \eqref{free.fp.R} (via standard ODE theory) admits a unique solution $R\in C^1([0,T])$ (notice that $g\in C^0$) which is non-decreasing and such that $R(t)\geq R_0>0$ for $t\geq 0$. Furthermore, 
given that $\rho_0(ry)\leq\rho_{cr}$ for $y\in\pa B$, $r>R_0$, it follows
\begin{align}\label{est.fp.g}
\|g_{\eps,\delta}[\tilde w](R(\cdot),\cdot)\|_{L^\infty(0,T)}\leq C(\eps)\|\kappa_\delta\|_{L^2(0,T)}
\|\na\tilde w\|_{L^2(0,T; L^2(\pa B))}\leq C(\eps,\delta)\|\tilde w\|_X.
\end{align}
So standard parabolic regularity theory \cite{Lie1996} yields the existence of a unique solution $w$ to \eqref{free.fp}--\eqref{free.fp.ic} such that 
$w\in L^2(0,T; H^{2}(B))$, $\pa_t w\in L^2(0,T; L^{2}(B))$.

We wish to derive a bound for $w$.
Multiplying \eqref{free.fp} by $-\Delta w$ and integrating in $B\times (0,t)$ yield
\begin{align*}
-\int_0^t\int_B \Delta w\, \pa_t w dy d\tau 
=& -s\int_0^t R(\tau)^{-2}g[\tilde w](R(\tau),\tau)\int_B \Delta w\, y\cdot\nabla w dy d\tau \\
&-\int_0^t\int_B (s [R(\tau)^{-2}]_\eps [a(\tilde{w})]_\eps + 1-s) |\Delta w|^2 dy d\tau.
\end{align*}
We cannot simply integrate by parts the integral on the left-hand side since we only know that $\pa_t w\in L^2(0,T; L^2(B))$. Therefore we consider, for $\tau>0$ small enough
\begin{align*}
-\tau^{-1}&\int_\tau^{t}\int_B \Delta w(t')\, (w(t')-w(t'-\tau)) dy dt' \\
&=\tau^{-1}\int_\tau^{t}\int_B \na w(t')\cdot
 (\na w(t')-\na w(t'-\tau)) dy dt' \\
&\geq \frac{1}{2\tau}\int_\tau^t\int_B |\na w(t')|^2 dy dt'
-\frac{1}{2\tau}\int_\tau^t\int_B |\na w(t'-\tau)|^2 dy dt'\\
&= \frac{1}{2\tau}\int_\tau^t\int_B |\na w(t')|^2 dy dt'
-\frac{1}{2\tau}\int_0^{t-\tau}\int_B |\na w(t')|^2 dy dt'\\
&= \frac{1}{2\tau}\int_{t-\tau}^t\int_B |\na w(t')|^2 dy dt'
-\frac{1}{2\tau}\int_0^\tau\int_B |\na w(t')|^2 dy dt'.
\end{align*}
Since $w\in L^2(0,T; H^{2}(B))$, $\pa_t w\in L^2(0,T; L^{2}(B))$, it also holds that $w\in C^0([0,T],H^1(B))$, thus we can apply the fundamental theorem of calculus to infer that
\begin{align*}
-\int_0^t\int_B \Delta w\, \pa_t w dy d\tau \geq 
\frac{1}{2}\int_B |\na w(t)|^2 dy - \frac{1}{2}\int_B |\na w(0)|^2 dy.
\end{align*}
Therefore, we obtain
\begin{align*}
\frac{1}{2}\int_B |\nabla w(t)|^2 dy 
&+\int_0^t\int_B (s [R(t)^{-2}]_\eps [a(\tilde{w})]_\eps + 1-s) |\Delta w|^2 dy d\tau
\leq \frac{1}{2}\int_B |\nabla w(0)|^2 dy \\
&-s\int_0^t R(\tau)^{-2}g[\tilde w](R(\tau),\tau)\int_B \Delta w\, y\cdot\nabla w dy d\tau.
\end{align*}
Furthermore, integration by parts lead to
\begin{align*}
\int_B \Delta w\, y\cdot\nabla w\, dy
=& \sum_{i,j=1}^{2}\int_B \pa_{y_j y_j}^2 w\, y_i\pa_{y_i}w\, dy\\
=& -\sum_{i,j=1}^{2}\int_B y_i\pa_{y_i y_j y_j}^3 w\, w\, dy
-2\sum_{j=1}^{2}\int_B w \pa_{y_j y_j}^2 w\, dy\\
=& \sum_{i,j=1}^{2}\int_B \pa_{y_i y_j}^2 w\, (\delta_{ij}w + y_i\pa_{y_j}w) dy - 2\int_B w\Delta w\, dy\\
=& -\int_B w\Delta w\, dy 
+ \int_B (y\cdot\nabla)\nabla w\cdot\nabla w\, dy\\
=& \int_B |\nabla w|^2 dy + \frac{1}{2}\int_B (y\cdot\nabla)(|\nabla w|^2) dy\\
=& \frac{1}{2}\int_{\pa B}|\na w|^2 d\sigma_y .
\end{align*}
Therefore, we get
\begin{align}\label{fp.est.1a}
\int_B |\nabla w(t)|^2 dy 
&+\int_0^t\int_B (s [R(t)^{-2}]_\eps [a(\tilde{w})]_\eps + 1-s)|\Delta w|^2 dy d\tau\\ \nonumber
&+s\int_0^t\int_{\pa B}R(\tau)^{-2}g[\tilde w](R(\tau),\tau)|\na w|^2 d\sigma_y d\tau
\leq \int_B |\nabla w(0)|^2 dy,\quad t\in [0,T]. 
\end{align}
Estimate \eqref{fp.est.1a} yields a uniform (in $s$) bound for $w$ in
$L^\infty(0,T; H^1(B)) \cap L^2(0,T; H^2(B))$.
Via this bound and \eqref{free.fp}, \eqref{est.fp.g} we get a $s-$uniform bound for $\pa_t w$ in $L^2(0,T; L^2(B))$. 
Being the embedding $H^{2}(B)\hookrightarrow H^{3/2}(B)$ compact and the embedding $H^{3/2}(B)\hookrightarrow L^{2}(B)$ continuous, Aubin-Lions Lemma yields that the embedding 
\begin{align*}
\{ w\in L^2(0,T; H^2(B))~:~ \pa_t w \in L^2(0,T; L^2(B)) \}
\hookrightarrow 
L^2(0,T; H^{3/2}(B)) = X
\end{align*}
is compact. It follows that $S: X\times [0,1]\to X$ is compact. The continuity of $S$ follows from standard arguments in parabolic regularity theory. Also, $S(\cdot,0):X\to X$ is trivially constant. Finally, we want to prove that the set
\begin{align*}
\{w\in X~:~\exists s\in [0,1]:~~ S(w,s)=w\}
\end{align*}
is bounded in $X$. Assume that $w\in X$, $s\in [0,1]$ satisfy $w=S(w,s)$. 
An $s-$uniform bound for $w$ in $X$ follows from the $L^2(0,T; H^2(B))$ bound for $w$ that is provided by \eqref{fp.est.1a}. Therefore, by Leray-Schauder's theorem we infer the existence of a fixed point $w_{\eps,\delta}\in X$ for $S(\cdot,1)$, that is, a solution to the regularized problem \eqref{free.1.eps}--\eqref{free.1.R.eps}. Furthermore $w_{\eps,\delta}$ is radially symmetric, positive inside $B\times (0,T)$ and uniformly (in $\eps$) bounded in $L^\infty(B\times (0,T))$ thanks to the strong maximum principle, and it is a solution to
\eqref{free.1.eps}--\eqref{free.1.R.eps}.
Moreover \eqref{fp.est.1a} holds with $s=1$ and $\tilde w = w = w_{\eps,\delta}$, $R=R_{\eps,\delta}$:
\begin{align}\label{fp.est.1}
\int_B |\nabla w_{\eps,\delta}(t)|^2 dy 
&+\int_0^t\int_B [R_{\eps,\delta}(t)^{-2}]_\eps [a(w_{\eps,\delta})]_\eps |\Delta w_{\eps,\delta}|^2 dy d\tau\\ \nonumber
&+\int_0^t\int_{\pa B}R_{\eps,\delta}(\tau)^{-2}g[w_{\eps,\delta}](R_{\eps,\delta}(\tau),\tau)|\na w_{\eps,\delta}|^2 d\sigma_y d\tau\\
\nonumber
& \leq \int_B |\nabla w_{\eps,\delta}(0)|^2 dy,\quad t\in [0,T]. 
\end{align}
{\em Step 2: Limit $\delta\to 0$.}
Thanks to \eqref{kappa.def}, \eqref{g.map.eps.delta} and \eqref{fp.est.1} it holds
\begin{align*}
\|g_{\eps,\delta}[w_{\eps,\delta}](R_{\eps,\delta}(\cdot),\cdot)\|_{L^2(0,T)} &\leq \eps^{-1}\|\na w_{\eps,\delta}\|_{L^2(0,T; L^1(\pa B))}\leq C\eps^{-1}\|\na w_{\eps,\delta}\|_{L^2(0,T; H^{1/2}(B))}\\
&\leq 
C\eps^{-1}\|w_{\eps,\delta}\|_{L^2(0,T; H^{2}(B))}\leq C(\eps).
\end{align*}
As a consequence of the above estimate and \eqref{free.1.R.eps}, we deduce that 
\begin{align*}
0<R_0\leq R_{\eps,\delta}(t)\leq C(\eps)(1+T^{1/4})\quad 0<t<T.
\end{align*}
Also, \eqref{fp.est.1} yields a $\delta-$uniform bound for $w_{\eps,\delta}$ in $L^2(0,T; H^2(B))\cap L^\infty(0,T; H^1(B))$. Together with the $\delta-$uniform $L^2(0,T)$ bound for $g_{\eps,\delta}(R(\cdot),\cdot)$, this allows us to obtain a $\delta-$uniform bound for $\pa_t w_{\eps,\delta}$ in $L^p(0,T; L^p(B))$ for some $p>1$ as well as a $\delta-$uniform bound for $R_{\eps,\delta}$ in ${H^1}(0,T)$. Via Aubin-Lions Lemma and the compact Sobolev's embedding ${H^1}(0,T)\hookrightarrow C^{0,\alpha}([0,T])$ ($0<\alpha<{1/2}$) we deduce that, up to subsequences,
\begin{align*}
w_{\eps,\delta}\to w_\eps\quad\mbox{strongly in $C^0([0,T], L^q(B))$, for every }2\leq q < 
\begin{cases}
\infty & d=2\\ \frac{2d}{d-2} & d\geq 3
\end{cases}\\
w_{\eps,\delta}\rightharpoonup w_\eps\quad\mbox{weakly in }L^2(0,T; H^2(B)),\\
w_{\eps,\delta}\rightharpoonup^* w_\eps\quad\mbox{weakly* in }L^\infty(0,T; H^1(B)),\\
R_{\eps,\delta}\to R_\eps\quad\mbox{strongly in $C^{0,\alpha}([0,T])$, for every $0<\alpha<{1/2}$,}\\
R_{\eps,\delta}\rightharpoonup R_\eps\quad\mbox{weakly in }{H^1}(0,T).
\end{align*} 
These relations allow us to take the limit $\delta\to 0$ in \eqref{free.1.eps}--\eqref{free.1.R.eps} in a rather straightforward way and obtain a solution $w_\eps\in L^2(0,T; H^2(B))\cap L^\infty(0,T; H^1(B))$, $\pa_t w_\eps\in L^2(0,T; L^2(B))$, $R_\eps\in {H^1}(0,T)$ to
\begin{align}\label{free.3.eps}
&\pa_t w_{\eps} = R_{\eps}(t)^{-2}g_{\eps}[w_\eps](R_{\eps}(t),t)y\cdot\nabla_y w_{\eps} + [R_{\eps}(t)^{-2}]_\eps [a(w_{\eps})]_\eps  \Delta_y w_{\eps}\quad (y,t)\in B\times (0,T),\\
\label{free.3.bc.eps}
& w_{\eps}(y,t) = 0\quad (y,t)\in\pa B\times (0,T),\\
\label{free.3.ic.eps}
& w_{\eps}(y,0) = f(\rho_0(R_0 y))\quad y\in B,\\
\label{free.3.R.eps}
& R_{\eps}'(t) = \frac{g_{\eps}[w_\eps](R_{\eps}(t),t)}{R_{\eps}(t)}\quad t>0,\quad R(0)=R_0,
\end{align}
where $g_\eps$ is defined in \eqref{g.map.eps}.\medskip\\
{\em Step 3: Limit $\eps\to 0$.}
We will now derive estimates for $(w_{\eps},R_\eps)$, the solution to \eqref{free.3.eps}--\eqref{free.3.R.eps}, which are uniform in $\eps$.
We already have an $\eps-$uniform bound for $w_\eps$ in $L^\infty(0,T; H^1(B))$. Our next goal is to derive an upper bound for $R_\eps$. In order to to achieve this goal, we preliminary find an estimate for the $L^1(\pa B\times (0,T))$ norm of $\na w_\eps$ in terms of $R_\eps$.

Let us consider the following estimate, contained in \eqref{fp.est.1}:
\begin{align*}
\int_0^t R_\eps(\tau)^{-2}g_\eps[w_\eps](R_\eps(\tau),\tau)\int_{\pa B}|\na w_\eps(y,\tau)|^2 d\sigma_y d\tau\leq C.
\end{align*}
Jensen's inequality yields
\begin{align*}
\int_0^t R_\eps(\tau)^{-2}g_\eps[w_\eps](R_\eps(\tau),\tau)
\left( \int_{\pa B}|\na w_\eps(y,\tau)| d\sigma_y\right)^2 d\tau
&\leq C ,
\end{align*}
and being $t\mapsto R_\eps(t)$ non-decreasing we get
\begin{align*}
\int_0^t g_\eps[w_\eps](R_\eps(\tau),\tau)
\left( \int_{\pa B}|\na w_\eps(y,\tau)| d\sigma_y\right)^2 d\tau
&\leq C R_{\eps}(t)^2 .
\end{align*}
However, from \eqref{g.map.eps} it follows
\begin{align*}
g_{\eps}[w_\eps](R_\eps(\tau),\tau)\geq |\pa B|^{-1}(1+\rho_{cr})^{-1}\int_{\pa B}|\na w_\eps(y,\tau)|d\sigma_y.
\end{align*}
Therefore, we get
\begin{align}\label{est.naw.paB.0}
\int_0^t \left( \int_{\pa B}|\na w_\eps(y,\tau)| d\sigma_y\right)^3 d\tau
\leq C R_\eps(t)^2
\end{align}
and applying Jensen's inequality once again yields
\begin{align}\label{est.naw.paB}
\int_0^t\int_{\pa B}|\na w_\eps(y,\tau)|d\sigma_y d\tau \leq C \,
(R_\eps(t)\, t)^{2/3}.
\end{align}
We now derive an upper bound for $R_\eps$.
Recall that $\rho_0$ is radially symmetric, thus we can write $\rho_0(y) = \overline{\rho_0}(|y|)$ for $y\in\overline{B}$.
Multiplying the ODE in \eqref{free.3.R.eps} by $R_\eps(t)(\rho_{cr}+\eps - \overline{\rho}_0(R_\eps(t)))$ and employing \eqref{g.map.eps} leads to
\begin{align*}
R_\eps(t)(\rho_{cr}+\eps - \overline{\rho}_0(R_\eps(t)))\pa_t R_\eps(t) = \fint_{\pa B}|\na w_\eps(y,t)|d\sigma_y
\end{align*}
and a time integration yields
\begin{align*}
G_\eps(R_\eps(t)) = \int_0^t\fint_{\pa B}|\na w_\eps(y,\tau)|d\sigma_y d\tau,\quad 0<t<T,\\
\nonumber
G_\eps(r) := \int_{R_0}^r s(\rho_{cr}+\eps - \overline{\rho}_0(s))ds.
\end{align*}
From \eqref{est.naw.paB} we infer
\begin{align}\label{COVID.1}
G_\eps(R_\eps(t)) \leq C (R_\eps(t)\, t)^{2/3},\quad 0<t<T.
\end{align}
Now we must distinguish two cases according to the value of
\begin{align*}
\mathfrak{R}(T)\equiv \liminf_{\eps\to 0}R_\eps(T)\geq R_0.
\end{align*}
{\em Case 1: $\mathfrak{R}(T)>R_0$.} In this case a positive constant $\eta(T)>0$ exists such that $R_\eps(T)\geq R_0 + \eta(T)$ for $\eps\to 0$. Given the assumption on the initial datum, it follows
\begin{align*}
\rho_{cr} - \overline{\rho_{0}}(s) \geq 
\rho_{cr} - \overline{\rho_{0}}(R_\eps(T)-\eta(T)/2)\geq 
\xi(T)>0,\quad R_\eps(T)-\eta(T)/2\leq s\leq R_\eps(T),
\end{align*}
as $\eps\to 0$. Thus, for some positive $\xi(T)$ it holds
\begin{align*}
G_\eps(R_\eps(T))\geq \frac{1}{2}\xi(T)( R_\eps(T)^2 - (R_\eps(T)-\eta(T)/2)^2) = 
\frac{1}{2}\xi(T)\eta(T)R_\eps(T) - \frac{1}{8}\xi(T)\eta(T)^2,
\end{align*}
as $\eps\to 0$. From \eqref{COVID.1} and Young's inequality (as well as the monotonicity of $R_\eps(t)$) we obtain
\begin{align*}
R_\eps(t)\leq C(T),\quad 0\leq t\leq T,\quad\eps\to 0.
\end{align*}
{\em Case 2: $\mathfrak{R}(T)=R_0$.} 
Since we are arguing up to subsequences, we can assume w.l.o.g.~that
$\lim_{\eps\to 0}R_\eps(T)=R_0$. In particular a constant $c_R(T)>0$ exists such that $R_\eps(T)\leq C(T)$ for $\eps>0$ small enough.
The monotonicity of $R_\eps(t)$ yields
\begin{align*}
R_\eps(t) \leq C(T)\quad 0\leq t\leq T,\quad \eps\to 0.
\end{align*}
In any case, we have proved that
\begin{align}\label{ub.R}
R_0\leq R_\eps(t)\leq C(T)\quad 0\leq t\leq T,\quad \eps\to 0.
\end{align}
We must now find a bound for $g_\eps[w_\eps](R_\eps(\cdot),\cdot)$.
To do so, we multiply \eqref{free.3.R.eps} by
$(\rho_{cr}+\eps-\overline{\rho_0}(R_\eps(t)))^{\lambda-1}$, with $0<\lambda<1$ arbitrary, and integrate on $[0,T]$:
\begin{align*}
\int_0^T \frac{g_{\eps}[w_\eps](R_{\eps}(t),t)}{R_{\eps}(t)
(\rho_{cr}+\eps-\overline{\rho_0}(R_\eps(t)))^{1-\lambda}}dt 
&= 
\int_0^T R_\eps'(t) (\rho_{cr}+\eps-\overline{\rho_0}(R_\eps(t)))^{\lambda-1}dt\\
&=\int_{R_0}^{R_\eps(T)}
(\rho_{cr}+\eps-\overline{\rho_0}(r))^{\lambda-1}dr,
\end{align*}
where the change of variables $r=R_\eps(t)$ in the integral is justified by the monotonicity of $R_\eps$. From \eqref{ub.R} and the definition \eqref{g.map.eps} of $g_\eps$ we deduce
\begin{align*}
\int_0^T\int_{\pa B} \frac{|\na_y w_\eps(y,t)|}{
(\rho_{cr}+\eps-\overline{\rho_0}(R_\eps(t)))^{2-\lambda}}d\sigma_y dt 
&\leq C(T)\int_{R_0}^{C(T)}
(\rho_{cr}-\overline{\rho_0}(r))^{\lambda-1}dr .
\end{align*}
Thanks to the assumptions on the initial datum the singularity inside the integral on the right-hand side of the above inequality is integrable, thus we get
\begin{align}\label{COVID.2}
\int_0^T\int_{\pa B} \frac{|\na_y w_\eps(y,t)|}{
	(\rho_{cr}+\eps-\overline{\rho_0}(R_\eps(t)))^{2-\lambda}}d\sigma_y dt 
&\leq C(T,\lambda),\quad 0<\lambda<1.
\end{align}
By writing
\begin{align*}
g_\eps[w_\eps](R_\eps(t),t) = 
(\rho_{cr}+\eps-\overline{\rho_0}(R_\eps(t)))^{-1}
\left( \int_{\pa B}|\na_y w_\eps|d\sigma_y\right)^{\frac{1}{2-\lambda}}
\cdot
\left( \int_{\pa B}|\na_y w_\eps|d\sigma_y\right)^{\frac{1-\lambda}{2-\lambda}},
\end{align*}
via H\"older's inequality we obtain
\begin{align*}
\|& g_\eps[w_\eps](R_\eps(\cdot),\cdot)\|_{L^p(0,T)}\\
&\leq 
\left\| (\rho_{cr}+\eps-\overline{\rho_0}(R_\eps(t)))^{-1}
\left( \int_{\pa B}|\na_y w_\eps|d\sigma_y\right)^{\frac{1}{2-\lambda}} \right\|_{L^{2-\lambda}(0,T)}
\left\|
\left( \int_{\pa B}|\na_y w_\eps|d\sigma_y\right)^{\frac{1-\lambda}{2-\lambda}}
\right\|_{L^{\frac{3(2-\lambda)}{1-\lambda}}(0,T)}
\end{align*}
with $\frac{1}{p} = \frac{1}{2-\lambda} + \frac{1-\lambda}{3(2-\lambda)}$, that is $p = p(\lambda) = \frac{3(2-\lambda)}{4-\lambda}$. Since $\lambda\in (0,1)\mapsto p(\lambda)\in (0,\infty)$ is decreasing and $p(0)=3/2$, 
from \eqref{est.naw.paB.0} and \eqref{COVID.2} we deduce
\begin{align}\label{est.g.eps}
\| g_\eps[w_\eps](R_\eps(\cdot),\cdot)\|_{L^{3/2 - \theta}(0,T)}\leq C(T,\theta),\quad 0<\theta\leq \frac{1}{2}.
\end{align}
Bound \eqref{est.g.eps} and the uniform lower bound on $R_\eps$ immediately yield the following estimate for $R_\eps'$ via \eqref{free.3.R.eps}
\begin{align}\label{est.dR.eps}
\| R_\eps'\|_{L^{3/2 - \theta}(0,T)}\leq C(T,\theta),\quad 0<\theta\leq \frac{1}{2}.
\end{align}
Another bound for $\na w_\eps$ in $L^{3}(0,T; L^{3/2}(\pa B))$ can easily be stated from \eqref{fp.est.1} as follows. $L^p$ interpolation leads to
\begin{align*}
\int_{0}^t \|\na w_\eps\|_{L^{3/2}(\pa B)}^3 d\tau \leq 
\int_{0}^t \|\na w_\eps\|_{L^{1}(\pa B)}
\|\na w_\eps\|_{L^{2}(\pa B)}^2 d\tau.
\end{align*}
From \eqref{g.map.eps} we obtain
\begin{align*}
\int_{0}^t \|\na w_\eps\|_{L^{3/2}(\pa B)}^3 d\tau \leq C
\int_0^t g_\eps[w_\eps](R_\eps(\tau),\tau)\|\na w_\eps\|_{L^{2}(\pa B)}^2 d\tau
\end{align*}
and \eqref{ub.R} yields
\begin{align*}
\int_{0}^t \|\na w_\eps\|_{L^{3/2}(\pa B)}^3 d\tau \leq C(T)
\int_0^t R_\eps(\tau)^{-2} g_\eps[w_\eps](R_\eps(\tau),\tau)\|\na w_\eps\|_{L^{2}(\pa B)}^2 d\tau,\quad 0\leq t\leq T.
\end{align*}
From \eqref{fp.est.1} we conclude that
\begin{align}\label{ub.naw.paB}
\|\na w_\eps\|_{L^3(0,T; L^{3/2}(\pa B))}\leq C.
\end{align}
Also, \eqref{fp.est.1}, \eqref{ub.R} and the $\eps-$uniform $L^\infty$ bounds for $w_\eps$ imply the following 
$\eps-$uniform bounds for $w_\eps$: 
\begin{align}\label{b.w}
\|w_\eps\|_{L^\infty(0,T; H^1(B))} + \|a(w_\eps)^{1/2}\Delta w_\eps\|_{L^2(0,T; L^2(B))} + \sqrt{\eps}\|\Delta w_\eps\|_{L^2(0,T; L^2(B))}
\leq C.
\end{align}
From \eqref{free.3.eps}, \eqref{est.g.eps} and \eqref{b.w} we deduce in a straightforward way
\begin{align}\label{b.wt}
\|\pa_t w_\eps\|_{L^{3/2-\theta}(0,T; L^2(B))}\leq C(T,\theta),\quad\forall 0<\theta\leq\frac{1}{2}.
\end{align}
Therefore Aubin-Lions Lemma yields the relative compactness of $w_\eps$ in $C^0([0,T]; L^2(B))$ (more general in $C^0([t_0,T]; L^p(B))$ for every $p<\infty$). Therefore, up to subsequences (not relabeled)
\begin{align*}
w_\eps\to w\quad \mbox{strongly in $C^0([0,T]; L^q(B))$ for every }
2\leq q<\begin{cases}
\infty & d=2\\ \frac{2d}{d-2} & d\geq 3
\end{cases}\\
w_\eps\rightharpoonup^* w \quad \mbox{weakly* in $L^\infty(0,T; H^1(B))$,}\\
w_\eps\rightharpoonup w\quad \mbox{weakly in $L^3(0,T; W^{1,3/2}(\pa B))$.}
\end{align*}
Concerning $R_\eps$, it holds
\begin{align*}
& 0 < R_0\leq R_\eps(t)\leq C(T)\qquad 0\leq t\leq T,\\
& |R_\eps(t_1)^2 - R_\eps(t_2)^2| = 2\int_{t_1}^{t_2}g_\eps[w_\eps](R_\eps(\tau),\tau)d\tau \leq
C(T,\xi)|t_1-t_2|^{1/3-\xi}\\
&\qquad\qquad 0\leq t_1 < t_2\leq T,\quad 0<\xi<\frac{1}{3}.
\end{align*}
Ascoli-Arzela's theorem yields the strong compactness of $(R_\eps)_{\eps>0}$ in $C^0([0,T])$. Thus, we can find a subsequence (not relabeled) of $R_\eps$ such that
\begin{align}\label{cnv.Reps}
R_\eps \to R\quad \mbox{strongly in $C^{0,\alpha}([0,T])$, ~~ $0<\alpha<1/3$.}
\end{align}
We define the new variable
\begin{align*}
v_\eps = \rho_{cr} + \int_0^{w_\eps} [a(s)]_\eps^{-1}ds.
\end{align*}
We also define
\begin{align*}
v := \rho_{cr} + \int_0^{w} a(s)^{-1}ds = 
\rho_{cr} + \int_0^{w} \frac{1}{f'(f^{-1}(s))}ds = 
\rho_{cr} + \int_{\rho_{cr}}^{f^{-1}(w)}d\sigma = f^{-1}(w).
\end{align*}
Consequently, we can write
\begin{align*}
w_\eps = f_\eps(v_\eps),\quad 
f_\eps = \mbox{inverse mapping of }s\in [0,\infty)\mapsto \rho_{cr} + \int_0^s [a(s')]_\eps^{-1}ds' \in [\rho_{cr},\infty).
\end{align*}
Dividing \eqref{free.3.eps} by $[a(w_\eps)]_\eps$ yields
\begin{align}\label{eq.veps}
\pa_t v_\eps = R_\eps(t)^{-2}g_\eps[w_\eps](R_\eps(t),t)y\cdot\na v_\eps
+ R_\eps(t)^{-2}\Delta f_\eps(v_\eps),\quad x\in B,~~ t\in (0,T),
\end{align}
where we point out that $[R_\eps(t)^{-2}]_\eps = R_\eps(t)^{-2}$ for 
$\eps>0$ small enough thanks to \eqref{ub.R}.
The strong convergence relations and the uniform $L^\infty$ bounds for $w_\eps$ imply 
\begin{align*}
|v_\eps - v| &\leq \int_{\min\{w_\eps,w\}}^{\max\{w_\eps,w\}}a(s)^{-1}ds
+ \int_0^{w_\eps}( a(s)^{-1} - [a(s)]_\eps^{-1} )ds\\
&\leq f^{-1}(\max\{w_\eps,w\}) - f^{-1}(\min\{w_\eps,w\})
+ \int_0^{ \min\{w_\eps,\eps\} }a(s)^{-1}ds\\
&\to 0\mbox{ a.e.~in }B\times (0,T),
\end{align*}
and therefore
\begin{align*}
v_\eps \to v\quad \mbox{strongly in }L^p(B\times (0,T)),~~\forall p<\infty.
\end{align*}
We now need to study the limit of the term $g_\eps[w_\eps](R_\eps(t),t)$.
It is straightforward to prove that $w_\eps$ is radially symmetric. Indeed, this property follows from the symmetry of \eqref{free.3.eps}, the initial datum and the uniqueness of solutions of \eqref{free.3.eps} (which can be easily shown). We also claim that the outward radial derivative of $w_\eps$ at $\pa B$ is non-positive. Indeed, if such derivative was positive, since $w_\eps=0$ at $\pa B$, this would imply $w_\eps<0$ at some point inside $B$ (notice that $y\in B\mapsto w_\eps(y,t)\in\R$ is continuous since $w_\eps\in C^0((0,T]; H^1(B))$ and radially symmetric), which is not the case. 
It follows
\begin{align}\label{id.mod}
|\na w_\eps(y,t)| = -y\cdot\na w_\eps(y,t)\quad y\in\pa B,~~t>0.
\end{align}
Therefore,
\begin{align*}
g_\eps[w_\eps](R_\eps(t),t) \rightharpoonup\int_{\pa B}\frac{-y\cdot\na w(y,t)}{\rho_{cr}-\rho_0(R(t)y)}dy\qquad\mbox{weakly in $L^{3/2-\theta}(0,T)$ for every $\theta\in (0,1/2]$.}
\end{align*}
Being $w\in C^0([0,T]; H^1(B))$ and radially symmetric, $w(t)$ is also continuous in $B$ for $t\in [0,T]$, and its outward radial derivative at $\pa B$ is non-positive, just like $w_\eps$. Thus,
\begin{align*}
\int_{\pa B}\frac{-y\cdot\na w(y,t)}{\rho_{cr}-\rho_0(R(t)y)}dy = 
\int_{\pa B}\frac{|\na w(y,t)|}{\rho_{cr}-\rho_0(R(t)y)}dy =
g[w](R(t),t)\quad\mbox{for }t>0.
\end{align*}
At this point we test \eqref{eq.veps} against $R_\eps(t)^d\phi$ with
$\phi\in C^1_c(\overline{B}\times [0,T))$ arbitrary and obtain
\begin{align*}
\int_0^T\int_B v_\eps\pa_t(R_\eps(t)^d\phi) \, dy dt + 
\int_B \rho_{0}(R_0 y)R_0^d\phi(y,0)dy \\
=
\int_0^T\int_B R_\eps(t)^{d-2} g_\eps[w_\eps](R_\eps(t),t) v_{\eps}\, \Div_y(y\phi)dy dt\\
+ \int_0^T\int_B R_\eps(t)^{d-2}\na_y f_\eps(v_\eps)\cdot\na_y\phi dy dt
\\
- \int_0^T\int_{\pa B}R_\eps(t)^{d-2}\phi y\cdot\na_y f_\eps(v_\eps)d\sigma_y dt
\end{align*}
and employing \eqref{free.3.R.eps} yields
\begin{align*}
\int_0^T\int_B v_\eps R_\eps(t)^d\pa_t \phi \, dy dt + 
\int_B \rho_{0}(R_0 y)R_0^d\phi(y,0)dy =
\int_0^T\int_B R'_\eps(t) R_\eps(t)^{d-1} v_{\eps}\, y\cdot\na_y\phi\, dy dt\\
+ \int_0^T\int_B R_\eps(t)^{d-2}\na_y f_\eps(v_\eps)\cdot\na_y\phi dy dt
- \int_0^T\int_{\pa B}R_\eps(t)^{d-2}\phi y\cdot\na_y f_\eps(v_\eps)d\sigma_y dt.
\end{align*}
Now, we can pass to the limit in the above identity and obtain
\begin{align*}
\int_0^T\int_B v R(t)^d\pa_t \phi \, dy dt + 
\int_B \rho_{0}(R_0 y)R_0^d\phi(y,0)dy =
\int_0^T\int_B R'(t) R(t)^{d-1} v\, y\cdot\na_y\phi\, dy dt\\
+ \int_0^T\int_B R(t)^{d-2}\na_y f(v)\cdot\na_y\phi dy dt
- \int_0^T\int_{\pa B}R(t)^{d-2}\phi y\cdot\na_y f(v)d\sigma_y dt .
\end{align*}
Applying the change of variables $x = R(t)y$, $u(x,t)=v(y,t)=v(R(t)^{-1}x,t)$, $\psi(x,t)=\phi(y,t)=\phi(R(t)^{-1}x,t)$ and recalling that \eqref{id.mod} holds also for $w$ lead to
\begin{align*}
\int_0^T\int_{|x|<R(t)} u \pa_t \psi \, dx dt + 
\int_{|x|<R_0} \rho_{0}(x)\psi(x,0)dx \\
= \int_0^T\int_{|x|<R(t)} \na_x f(u)\cdot\na\psi dx dt
+ \int_0^T \int_{|x|=R(t)}\psi |\na_x f(u)| d\sigma_x dt .
\end{align*}
The limit in the ODE in \eqref{free.3.R.eps} is carried out by integrating the equation in time, i.e., 
\begin{align*}
R(t) = R_0 + \int_0^t\frac{g[w](R(\tau),\tau)}{R(\tau)}d\tau,\qquad t>0.
\end{align*}
This finishes the proof of the first part of Theorem \ref{thm.fb}.

\subsection{Relation between \eqref{fb}--\eqref{fb.ic} and
\eqref{1}--\eqref{2}.}
We are going to prove the following:
\begin{lemma}\label{lem.fb}
	Let $(u,R)$ be a weak solution to \eqref{fb}--\eqref{fb.ic} 
	according to the first part of Theorem~\ref{thm.fb}.
	Then, the function $\rho : \Omega\times (0,T)\to [0,\infty)$ defined as
	\begin{align}\label{rho.fb}
	\rho(t) = u(t)\chf{{E(t)}} + \rho_{0}\chf{\Omega\backslash {E(t)}}\quad t > 0,\quad E(t) = \{|x|\leq R(t)\},
	\end{align}
	is a weak solution to \eqref{1}--\eqref{2} according to Theorem~\ref{thm.existence}. 
	Furthermore, in the case that $\fint_\Omega\rho_0 dx < \rho_{cr}$,
	the steady state $\hat{\rho} = \lim_{t\to\infty}\rho(t)$ is given by
	\begin{align*}
	\hat{\rho} = \rho_{cr}\chf{{E_\infty}} + \rho_{0}\chf{\Omega\backslash {E_\infty}},\quad E_\infty = \{|x|\leq R_\infty\},\quad R_\infty = \lim_{t\to \infty}R(t).
	\end{align*}
\end{lemma}

\begin{proof}
	Notice that from the fact that 
	$\pa_t R\geq 0$ on $(0,\infty)$ it follows that 
	$E_0\subset E(t)$ for $t\geq 0$ and therefore $\rho_{0}< \rho_{cr}$ on $\Omega\backslash \overline{E(t)}$.
	It follows that ${E(t)} = \{\rho(t)>\rho_{cr}\}$ for $t\geq 0$.
	The assumptions on $u$ imply that \eqref{weaksol.1} holds. 
	We wish to prove that also \eqref{weaksol.2} is fulfilled.
	We first prove that the following integral identity 
	\begin{align}\label{freedom.1}
	-\int_0^T\int_\Omega &\rho \pa_t \phi dx dt - \int_\Omega \rho_{0}\phi(0) dx + \int_0^T\int_\Omega\nabla\phi\cdot\nabla f(\rho) dx dt \\ \nonumber 
	&= \int_0^T\int_{E(t)}\nabla\phi\cdot\nabla f(u) dx dt
	+\int_0^T \int_{\pa E(t)}\phi |\nabla f(u)|\, d\Sigma dt\\
	&\qquad \nonumber
	-\int_0^T\int_{E(t)} u \pa_t \phi dx dt
	-\int_{E_0}\rho_0 \phi(0)dx
	\end{align}
	holds for $\phi\in C^\infty_c(\Omega\times [0,T))$ arbitrary and $u : \mathcal{G}_T\to [0,\infty)$, $R : [0,T]\to [R_0,\infty)$ smooth enough satisfying just \eqref{fb.bc}--\eqref{fb.ic}.
	Since the right-hand side of \eqref{freedom.1} vanishes for weak solutions to \eqref{fb}--\eqref{fb.ic}, a density argument will then yield the statement.
	
	It holds
	\begin{align}\label{fb.1}
	-\int_0^T\int_\Omega &\rho \pa_t \phi dx dt - \int_\Omega \rho_{0}\phi(0) dx \\
	\nonumber
	=& 
	-\int_0^T\int_{E(t)} u \pa_t \phi dx dt 
	-\int_0^T\int_{\Omega\backslash E(t)} \rho_0 \pa_t \phi dx dt
	- \int_\Omega \rho_{0}\phi(0) dx\\
	\nonumber
	=& 
	-\int_0^T\int_{E(t)} (u-\rho_0) \pa_t \phi dx dt 
	-\int_0^T\int_{\Omega} \rho_0 \pa_t \phi dx dt
	- \int_\Omega \rho_{0}\phi(0) dx\\
	\nonumber
	=& 
	-\int_0^T\int_{E(t)} \pa_t [(u-\rho_0) \phi] dx dt 
	+\int_0^T\int_{E(t)} \phi \pa_t u dx dt \\
	\nonumber
	=& -\int_0^T\int_{E(t)} \pa_t [(u-\rho_0) \phi] dx dt 
	-\int_{E_0} \phi(0) \rho_0 dx dt
	-\int_0^T\int_{E(t)}u\pa_t\phi dx dt.
	\end{align}
	From the definition of $E(t)$ it follows
	\begin{align*}
	\frac{d}{dt}\int_{E(t)} (u-\rho_0)\phi dx &= 
	\frac{d}{dt}\int_{\pa B_1}\int_0^{R(t)}(u-\rho_0)\phi r^{d-1} dr d\sigma ,
	\end{align*}
	where $(r,\sigma)\in [0,\infty)\times \pa B_1$ are the spherical coordinates on $\R^d$. It follows
	\begin{align*}
	\frac{d}{dt}&\int_{E(t)} (u-\rho_0)\phi dx \\
	&= 
	\int_{\pa B_1}[(u-\rho_0)\phi]\vert_{r=R(t)} R(t)^{d-1}\pa_t R(t) d\sigma + 
	\int_{\pa B_1}\int_0^{R(t)}
	\pa_t[(u-\rho_0)\phi] r^{d-1} dr d\sigma\\
	&= 
	\int_{\pa B_1}[(u-\rho_0)\phi]\vert_{r=R(t)} R(t)^{d-1}\pa_t R(t) d\sigma + 
	\int_{E(t)} \pa_t [(u-\rho_0) \phi] dx .
	\end{align*}
	So from \eqref{fb.1} it follows
	\begin{align*}
	-\int_0^T\int_\Omega &\rho \pa_t \phi dx dt - \int_\Omega \rho_{0}\phi(0) dx \\
	=& 
	-\int_0^T \frac{d}{dt}\int_{E(t)} (u-\rho_0)\phi dx dt \\
	&+\int_0^T \int_{\pa B_1}[(u-\rho_0)\phi]\vert_{r=R(t)} R(t)^{d-1}\pa_t R(t) d\sigma dt
	-\int_{E_0} \phi(0) \rho_0 dx dt
	-\int_0^T\int_{E(t)}u\pa_t\phi dx dt\\
	=&\int_0^T \int_{\pa B_1}[(u-\rho_0)\phi]\vert_{r=R(t)} R(t)^{d-1}\pa_t R(t) d\sigma dt
	-\int_{E_0} \phi(0) \rho_0 dx dt
	-\int_0^T\int_{E(t)}u\pa_t\phi dx dt.
	\end{align*}
	The boundary conditions \eqref{fb.bc} imply
	\begin{align}\label{fb.2}
	-\int_0^T\int_\Omega &\rho \pa_t \phi dx dt - \int_\Omega \rho_{0}\phi(0) dx \\ \nonumber
	=& \int_0^T \int_{\pa E(t)}\phi
	|\nabla f(u)| d\Sigma dt
	-\int_{E_0} \phi(0) \rho_0 dx dt
	-\int_0^T\int_{E(t)}u\pa_t\phi dx dt.
	\end{align}
	Let us now compute
	\begin{align*}
	\int_0^T\int_\Omega\nabla\phi\cdot\nabla f(\rho) dx dt = 
	-\int_0^T\int_\Omega\Delta\phi\, f(\rho) dx dt =
	-\int_0^T\int_{E(t)}\Delta\phi\, f(u) dx dt ,
	\end{align*}
	where the last identity follows from the fact that $\rho_0 < \rho_{cr}$ on $\Omega\backslash E(t)$ for every $t\geq 0$.
	Employing the boundary conditions \eqref{fb.bc} yields
	\begin{align*}
	\int_0^T\int_\Omega &\nabla\phi\cdot\nabla f(\rho) dx dt 
	= \int_0^T\int_{E(t)}\nabla\phi\cdot\nabla f(u) dx dt.
	\end{align*}
	By summing the above identity with \eqref{fb.2} we get \eqref{freedom.1}. Thus $\rho$ is a weak solution to \eqref{1}--\eqref{2} in the sense of Theorem \ref{thm.existence}.
	Finally, if $\fint_\Omega\rho_0 dx < \rho_{cr}$,
	the expressions for $\hat{\rho}$ and $R_\infty$ follow from Theorem~\ref{thm.ltb.diff} and mass conservation, respectively.
	This finishes the proof of the lemma.
\end{proof}
The proof of Theorem \ref{thm.fb} is now concluded.

\section{Numerical results}\label{sec.numerics}

In this section, we discuss several numerical tests. The main goal is to observe the long time behavior of solutions as indicated in Theorem~\ref{thm.ltb.diff}
and segregation phenomena that are not covered by our theoretical investigations, but also observed in \cite{WGA2021}.
Since our intention is not the design of a better numerical scheme, we simply use the standard backward Euler and central finite difference scheme to 
discretize~\eqref{1}.

In the numerical tests, we take $\Omega = (-1,1)^2\in \R^2$ and assume $f$ to be
\begin{align*}
f(r) = (r-1)_+^2,\qquad r\geq 0.
\end{align*}
Therefore $\rho_{cr} = 1$ in this setting. 

\subsection{Numerical scheme.}
The numerical scheme reads as follows. An implicit Euler time discretization is employed:
\begin{align}\label{Euler}
\rho^{(n)} - \tau\Div (f'(\rho^{(n)})\nabla\rho^{(n)}) = \rho^{(n-1)},\quad n\geq 1,
\end{align}
where $\rho^{(0)}$ is the initial datum and $\rho^{(n)}$ is the solution at time $t_n = n\tau$. The spatial derivatives are discretized via centered finite differences:
\begin{align}\label{discr}
\begin{cases}
\rho^{(n)}_{ij} - 
\frac{\tau}{h_x}(J^{(1,n)}_{i+1/2,j} - J^{(1,n)}_{i-1/2,j}) -
\frac{\tau}{h_y}(J^{(2,n)}_{i,j+1/2} - J^{(2,n)}_{i,j-1/2})
 = \rho^{(n-1)}_{ij},\medskip\\
J_{ij}^{(1,n)} = \frac{1}{h_x}f'(\rho^{(n)}_{ij})(\rho^{(n)}_{i+1/2,j} - \rho^{(n)}_{i-1/2,j}),\medskip\\
J_{ij}^{(2,n)} = \frac{1}{h_y}f'(\rho^{(n)}_{ij})(\rho^{(n)}_{i,j+1/2} - \rho^{(n)}_{i,j-1/2}),
\end{cases}
\end{align}
where $\rho^{(n)}_{ij}$ is the value of $\rho^{(n)}$ at the point $(x_i,y_j)$ of the spatial grid, namely
\begin{align*}
(x_i,y_j) = (-1 + i h_x, -1 + j h_y),\quad i = 0,\ldots,N_x,~~
j = 0,\ldots,N_y,\quad h_x = \frac{2}{N_x},\quad h_y = \frac{2}{N_y}.
\end{align*}
At every time step, \eqref{discr} is solved via the following fix point scheme:
\begin{align}\label{scheme}
\begin{cases}
\rho^{(n,0)}_{ij} := \rho^{(n-1)}_{ij},\medskip\\
k\geq 1:~~
\begin{cases}
\rho^{(n,k)}_{ij} - 
\frac{\tau}{h_x}(J^{(1,n,k)}_{i+1/2,j} - J^{(1,n,k)}_{i-1/2,j}) -
\frac{\tau}{h_y}(J^{(2,n,k)}_{i,j+1/2} - J^{(2,n,k)}_{i,j-1/2})
= \rho^{(n-1)}_{ij},\medskip\\
J_{ij}^{(1,n,k)} = \frac{1}{h_x}f'(\rho^{(n,k-1)}_{ij})(\rho^{(n,k)}_{i+1/2,j} - \rho^{(n,k)}_{i-1/2,j}),\medskip\\
J_{ij}^{(2,n,k)} = \frac{1}{h_y}f'(\rho^{(n,k-1)}_{ij})(\rho^{(n,k)}_{i,j+1/2} - \rho^{(n,k)}_{i,j-1/2}).
\end{cases}
\end{cases}
\end{align}
Notice that \eqref{scheme} is just a linear system. 
\newline
The computation of the sequence $(\rho^{(n,k)}_{ij})_{k\geq 0}$ is stopped once the relative difference between two subsequent iterates is lower than a certain tolerance. The final iterate is defined as the approximate solution $\rho^{(n)}_{ij}$ to \eqref{discr}. Otherwise,
in case the number of fixed point iterations exceed a certain threshold without achieving convergence, the iteration is stopped and an error flag is returned, which is then used in the time iteration (see following part).
\newline
While the spatial grid is uniform and the number of intervals $N_x$, $N_y$ are fixed, the timestep $\tau$ is chosen at every time iteration in an adaptive way. Precisely, if the relative difference between the new and the old iterates is larger than a certain tolerance, or the fixed point procedure returned an error flag (see previous part), the timestep is reduced by a factor $1/2$, while if the relative difference between the new and the old iterates and the number of fixed point iterations are smaller than certain tolerances, then the timestep is increased by a factor $11/10$.

\subsection{Results.} We present here some simulation results employed with the scheme illustrated in the previous subsection. Three sets of results are presented, the first corresponding to the {\em supercritical} case (that is, the spatial average of the initial datum is larger than the critical value), the second and third corresponding to the {\em subcritical} case (that is, the spatial average of the initial datum is smaller than the critical value).
In both cases we observe that the evolution of the system is driven by the diffusion in the supercritical region (i.e.~the region of $\Omega$ where the solution is larger than $\rho_{cr}$) together with mass conservation. Indeed, while clearly the diffusion in the supercritical region pushes the mass downwards towards the critical density, this process drives the evolution of the solution in the whole spatial domain thanks to mass conservation: the mass distribution in the subcritical region (where the solution is smaller than $\rho_{cr}$) is pushed aside and upwards by the mass coming down from the supercritical region. We also observe that in the supercritical case these two processes are sufficiently strong to bring the whole mass distribution to the constant steady state $\rho_{\infty}$, while in the subcritical case the evolution of the system stops before achieving this convergence: the solution is simply pushed by the diffusion towards a non-constant profile which lies entirely in the subcritical region, and after such point the solution do not change any more. We remark that these phenomena can be already observed after a rather short simulation time.

In what follows $\fint_\Omega = |\Omega|^{-1}\int_\Omega$.
\medskip\\
{\bf Supercritical case.} We choose as initial datum
\begin{align*}
\rho^{(0)}(x) = \rho_{\infty} \frac{e^{-3|x|^2}}{\fint_\Omega e^{-3|y|^2} dy} \quad x\in\Omega,\qquad \rho_{\infty} = \frac{3}{2} > 1 = \rho_{cr}.
\end{align*}
The spatial domain discretization is a uniform grid of $401\times 401$ points.

\begin{figure}[h]
\includegraphics[width=\textwidth]{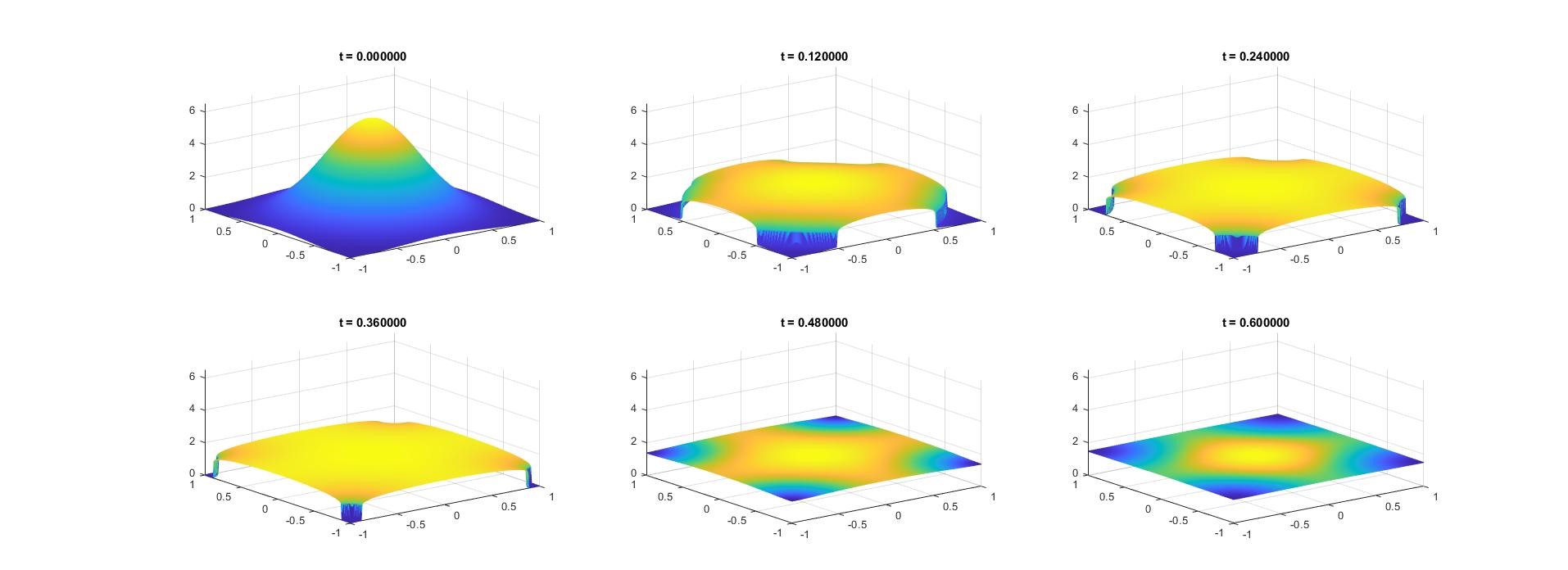}
\caption{Supercritical case. We observe convergence to the constant steady state 
$\rho_{\infty}$ of average mass.}
\label{fig1}
\end{figure}
In this case we observe that the solution converges towards the constant steady state given by the average mass $\rho_{\infty}$, see Figure~\ref{fig1}. The convergence is faster around the center of the mass distribution, where the solution takes the largest values, and it is slowest close to the vertexes of the rectangular domain $\Omega = [-1,1]^2$. While the solution remains very small around such points for some time, eventually the combination of diffusion in the supercritical region and mass conservation pushes the solution everywhere towards the constant steady state. Such a behavior is to be expected from the analysis result in Theorem \ref{thm.ltb.diff}.
\medskip\\
{\bf Subcritical case.} This is the most interesting case, as the analysis provides us with weaker results in this regime (namely, the steady state $\hat{\rho} = \lim_{t\to\infty}\rho(t)$ is unknown in general). We discretize now the spatial domain via a uniform grid with $501\times 501$ points (the increased number of grid points with respect to the supercritical case aims at obtaining a more accurate result for the more interesting subcritical case).
We present two examples.\medskip\\
{\em Example 1.} We choose as initial datum the sum of two Gaussians with the same variance: 
\begin{align*}
\rho_{in}(x)= \rho_{\infty}\frac{e^{-16|x-x^{(0)}|^2} + e^{-16|x+x^{(0)}|^2}}{\fint_\Omega (e^{-16|y-x^{(0)}|^2} + e^{-16|y+x^{(0)}|^2})dy}
\quad x\in\Omega,
\qquad x^{(0)} = \left(-\frac{7}{20}, \frac{7}{20}\right),\qquad
\rho_{\infty} = \frac{3}{4}.
\end{align*}
In Figure \ref{fig2}, we observe that the two mass distributions, which are quite separate at initial time, are pushed towards a joint mass distribution which roughly equals the critical density $\rho_{cr}=1$ on a large part of $\Omega$, while vanishes near the border of $\Omega$.
\begin{figure}[h]
\includegraphics[width=\textwidth]{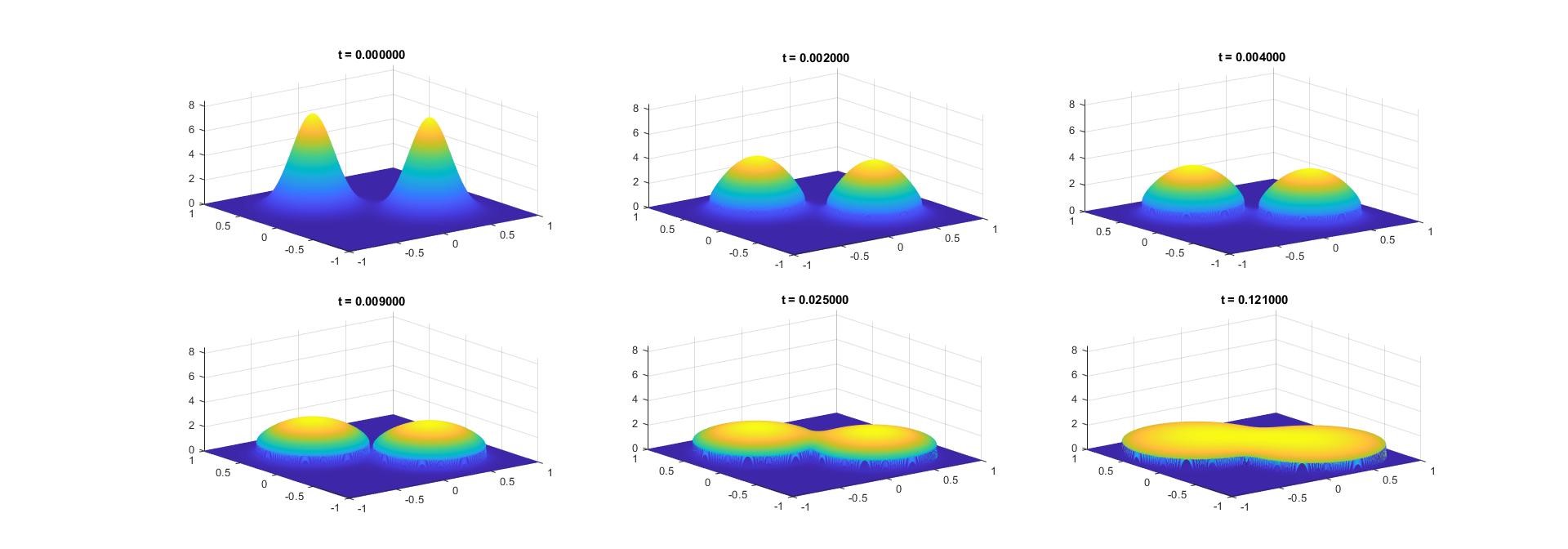}
\caption{Subcritical case, example 1. In the final configuration the two initial mass distributions merge into a single one.}
\label{fig2}
\end{figure}

{\em Example 2.} We choose again as initial datum the sum of two Gaussians with 
the same variance, but differently from the previous example, in this case the variance and distance between the centers are larger and the average density is smaller:
\begin{align*}
\rho_{in}(x)= \rho_{\infty}\frac{e^{-8|x-x^{(0)}|^2} + e^{-8|x+x^{(0)}|^2}}{\fint_\Omega (e^{-8|y-x^{(0)}|^2} + e^{-8|y+x^{(0)}|^2})dy}
\quad x\in\Omega,
\qquad x^{(0)} = \left(\frac{3}{4}, -\frac{3}{4}\right),\qquad
\rho_{\infty} = \frac{3}{10}.
\end{align*}
In Figure \ref{fig3}, we observe that this time the two mass distributions remain separated until the solution falls entirely underneath the critical value and the solution stops evolving. The final configuration of the system is given by two disjoint mass distributions separated by a valley.
\begin{figure}[h]
\includegraphics[width=\textwidth]{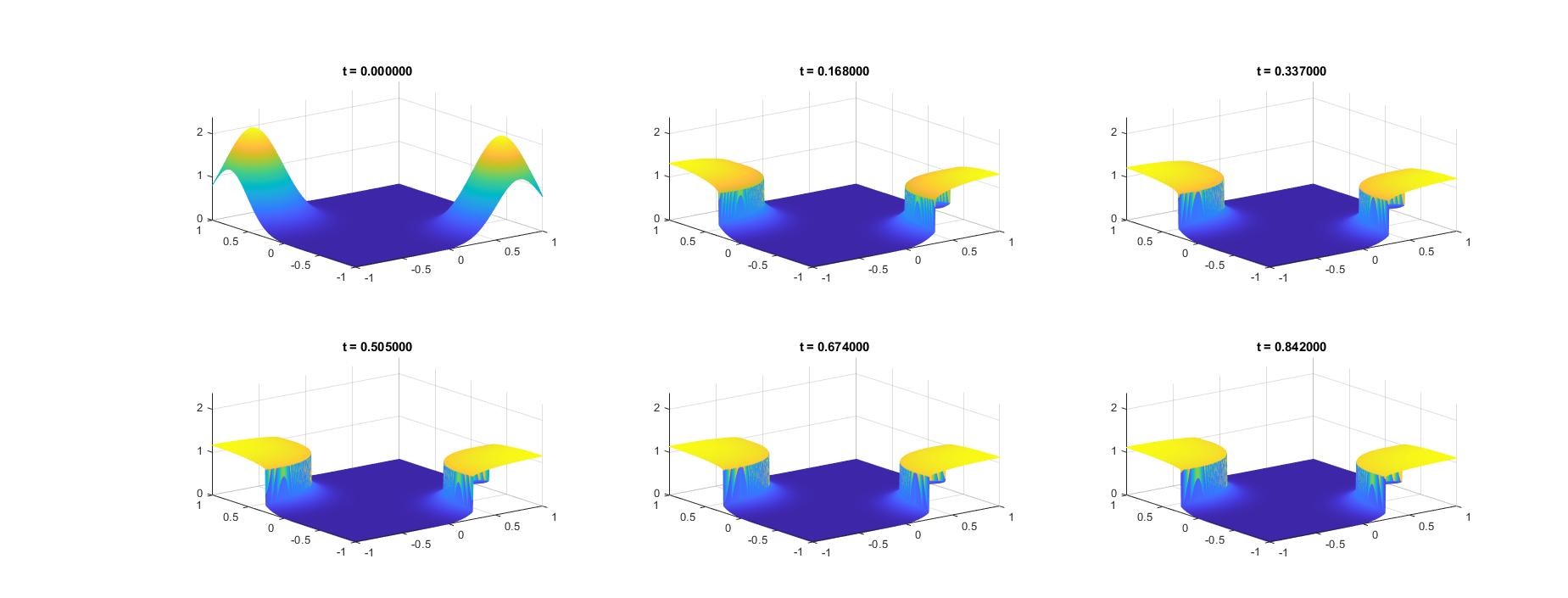}
\caption{Subcritical case, example 2. Final configuration is made up by two disjoint mass distributions separated by a valley.}
\label{fig3}
\end{figure}

\section{Conclusion}
In this paper we have proved the existence of a bounded unique global-in-time solution for a strongly degenerate parabolic problem modeling material flow problems and swarming behavior. Furthermore we have proved that the solution converges to a steady state, which is generally expected in diffusion equations.
Such steady state is equal to the average of the initial datum in the supercritical case, while in the subcritical case it is not known for general initial data, although we have showed that it is a.e.~not larger than the critical density. 
The convergence is proved to be exponential in the $L^2(\Omega)$ norm in the supercritical case, while in the subcritical case it is shown to be algebraic in a negative homogeneous Sobolev seminorm.
We also proved that, in the case of radially symmetric initial data which are decreasing in the radial coordinate, the strongly degenerate parabolic equation can be derived from a free boundary problem which only degenerates on the free boundary. Furthermore, we have proved global existence of weak solutions to the latter problem.
Numerical experiments for both cases supercritical and subcritical are carried out. In the subcritical case,
a segregation effect is observed when the distance between two high density regions is large enough, while no segregation happens when such distance is relatively small. This motivates future analytical projects aiming at determining criteria for initial data yielding the onset of segregation phenomena, and the investigation of how the solution generates segregation behavior on a broader viewpoint. The issue of determining the value of the steady state in the subcritical case is also worth investigating.
More broadly, it would also be interesting to analyze the qualitative behavior of the solutions, e.g.~whether the equation preserves the the solution segregation. Concerning this issue, we believe that energy methods would not really help in this investigation, since they appear to provide little to no information about the behavior and properties of the solution in the subcritical region. Instead, a possibly more fruitful approach would involve dealing with a free boundary problem associated to the degenerate parabolic equation, which would generalize system \eqref{fb}--\eqref{fb.ic} to the case of more general domains $\Omega$ and initial data $\rho_{0}$.

\section*{Appendix}

\subsection*{Auxiliary results.}
We prove here a generalized Poincar\'e's Lemma.
\begin{lemma}\label{lem.new.f}
	Let $\Omega$ be connected, $f$ as in \eqref{hp.f} and 
	$0\leq \rho_\infty\neq \rho_{cr}$. 
	For every $R>0$ a constant $C_R>0$ exists such that
	\begin{align*}
	\|f(\rho)-f(\rho_\infty)\|_{L^2(\Omega)}\leq C_R \|\na f(\rho)\|_{L^2(\Omega)}
	\end{align*}
	for every non-negative $\rho\in L^\infty(\Omega)$ such that $\na f(\rho)\in L^2(\Omega)$,
	$\|\rho\|_{L^\infty(\Omega)}\leq R$
	and $\fint_\Omega\rho dx = \rho_\infty$.
\end{lemma}
\begin{proof} By contradiction. Let $(\rho_n)_{n\in\N}$ be a sequence of non-negative bounded functions $\Omega\to\R$ such that 
	$$ \|\rho_n\|_{L^\infty(\Omega)}\leq R,\quad
	\langle\rho_n \rangle = \rho_\infty,\quad
	n\|\nabla f(\rho_n)\|_{L^2(\Omega)} < \|f(\rho_n)-f(\rho_\infty)\|_{L^2(\Omega)},\quad
	n\in\N . $$
	As a consequence $\|f(\rho_n)-f(\rho_\infty)\|_{L^2(\Omega)} > 0$ for every $n$, so we can define $u_n = (f(\rho_n)-f(\rho_\infty))/\|f(\rho_n)-f(\rho_\infty)\|_{L^2(\Omega)}$. It holds that 
	$\|\nabla u_n\|_{L^2(\Omega)} < 1/n$, so $\nabla u_n\to 0$ strongly in $L^2(\Omega)$. Since $\|u_n\|_{L^2(\Omega)}=1$, $u_n$ is bounded in $H^1(\Omega)$ and therefore via compact Sobolev embedding it is relatively compact in $L^2(\Omega)$, so up to subsequences $u_n\to u$ strongly in $L^2(\Omega)$ and weakly in $H^1(\Omega)$. Given that $\nabla u_n\to 0$ strongly in $L^2(\Omega)$ and $\Omega$ is connected, we deduce that $u$ is constant in $\Omega$. Since $\|u_n\|_{L^2(\Omega)}=1$ it follows that $u\neq 0$.
	As a consequence (up to subsequences)
	\begin{align}\label{apx.0}
	\frac{f(\rho_n)-f(\rho_\infty)}{\|f(\rho_n)-f(\rho_\infty)\|_{L^2(\Omega)}}\to u \quad\mbox{a.e.~in $\Omega$ and strongly in $L^2(\Omega)$.}
	\end{align}
	We distinguish two cases.\medskip\\
	{\em Case 1: $\rho_{\infty}>\rho_{cr}$.}
	Relation \eqref{apx.0} implies that, for every $\eps>0$ there exists $\Omega_\eps\subset\Omega$ such that $|\Omega\backslash\Omega_\eps|<\eps$ and
	\begin{align}\label{apx.1}
	\frac{f(\rho_n)-f(\rho_\infty)}{
		\|f(\rho_n)-f(\rho_\infty)\|_{L^2(\Omega)}}\to u\quad\mbox{strongly in }L^\infty(\Omega_\eps).
	\end{align}
	We show now that $\|f(\rho_n)-f(\rho_\infty)\|_{L^2(\Omega)}\to 0$ as $n\to\infty$. Indeed, if (by contradiction) there exists a subsequence (not relabeled) of $\rho_n$ such that
	$\|f(\rho_n)-f(\rho_\infty)\|_{L^2(\Omega)}\geq c>0$ for every $n\in\N$, then \eqref{apx.1} and the fact that $u$ is a non-zero constant imply the existence of $N_\eps>0$ such that
	\begin{align*}
	f(\rho_n)\geq f(\rho_\infty) + \frac{1}{2}cu\quad\mbox{if }u>0,\qquad
	f(\rho_n)\leq f(\rho_\infty) + \frac{1}{2}cu\quad\mbox{if }u<0,
	\end{align*}
	a.e.~in $\Omega_\eps$, for $n>N_\eps$. It follows that either $\rho_n \leq \rho_\infty - K$ a.e.~in $\Omega_\eps$, $n>N_\eps$, or 
	$\rho_n \geq \rho_\infty + K$ a.e.~in $\Omega_\eps$, $n>N_\eps$, for some $\eps-$independent constant $K>0$, which clearly violates the constraint $\fint_\Omega\rho_n dx = \rho_\infty$. Therefore,
	\begin{align}\label{apx.2}
	\|f(\rho_n)-f(\rho_\infty)\|_{L^2(\Omega)}\to 0\quad\mbox{as }n\to\infty.
	\end{align}
	From \eqref{apx.1}--\eqref{apx.2} we deduce
	\begin{align*}
	f(\rho_n)-f(\rho_\infty) \to 0\quad\mbox{strongly in }L^\infty(\Omega_\eps).
	\end{align*}
	Since $\rho_\infty>\rho_{cr}$ and $f$ is a one-to-one mapping on $(\rho_{cr},\infty)$, we easily deduce that
	\begin{align*}
	\rho_n\to \rho_\infty\quad\mbox{strongly in }L^\infty(\Omega_\eps),
	\end{align*}
	Being $|\Omega\backslash\Omega_\eps|<\eps$ and $\rho_n$ bounded in $L^\infty(\Omega)$ we obtain that
	\begin{align*}
	\rho_n\to \rho_\infty \quad\mbox{strongly in $L^p(\Omega)$ for every $p<\infty$,}
	\end{align*}
	and therefore
	\begin{align*}
	\frac{\rho_n - \rho_\infty}{f(\rho_n)-f(\rho_\infty)}\to \frac{1}{f'(\rho_\infty)}>0\quad\mbox{strongly in }L^2(\Omega).
	\end{align*}
	From the above relation and \eqref{apx.0} it follows
	\begin{align*}
	\frac{\rho_n-\rho_\infty}{\|f(\rho_n)-f(\rho_\infty)\|_{L^2(\Omega)}}\to \frac{u}{f'(\rho_\infty)}\neq 0 \quad\mbox{strongly in $L^1(\Omega)$,}
	\end{align*}
	which contradicts the constraint $\fint_\Omega\rho_n dx = \rho_\infty$.\medskip\\
	{\em Case 2: $\rho_\infty<\rho_{cr}$.} In this case $f(\rho_\infty)=0$ and $u>0$. We deduce from \eqref{apx.0}
	\begin{align*}
	\mbox{for a.e.~}x\in\Omega \quad \exists n^*(x)\in\N:
	\quad f(\rho_n(x)) > 0\quad\mbox{for }n\geq n^*(x).
	\end{align*}
	Given the definition of $f$, it follows
	\begin{align*}
	\mbox{for a.e.~}x\in\Omega \quad \exists n^*(x)\in\N:
	\quad \rho_n(x) > \rho_{cr}\quad\mbox{for }n\geq n^*(x).
	\end{align*}
	We deduce
	\begin{align*}
	\liminf_{n\to\infty}\rho_n \geq \rho_{cr}\quad\mbox{a.e.~in }\Omega.
	\end{align*}
	Fatou's Lemma and the assumption on the mass of $\rho_n$ yield
	\begin{align*}
	\rho_{cr} = \fint_\Omega \rho_{cr} dx\leq \fint_\Omega\liminf_{n\to\infty} \rho_n dx \leq \liminf_{n\to\infty}\fint_\Omega \rho_n dx = \rho_\infty
	\end{align*}
	against the assumption $\rho_\infty < \rho_{cr}$.
	
	This finishes the proof of the lemma.
\end{proof}


\begin{thebibliography}{99}
	
 \bibitem{AV2005} A. Amadori and J. L. Vazquez, Singular free boundary problem from image processing, Mathematical Models and Methods in Applied Sciences 15.05 (2005): 689-715.
 \bibitem{ACS1996} I. Athanasopoulos, L. Caffarelli, and S. Salsa, Regularity of the free boundary in parabolic phase-transition problems. Acta Mathematica 176.2 (1996): 245-282.
\bibitem{BV2006} F. Bachmann and J. Vovelle. Existence and uniqueness of entropy solution of scalar conservation laws with a flux function involving discontinuous coefficients. Communications in Partial Differential Equations 31(3) (2006): 371-395.	
\bibitem{BCBT1999} M. C. Bustos, F. Concha, R. B\"urger, and E.M. Tory, Sedimentation and thickening: phenomenological foundation and mathematical theory, Kluwer Academic Publishers, Dordrecht, The Netherlands, (1999).
	\bibitem{BFK2002} R. B\"urger, H. Frid, and K. H. Karlsen, On a free boundary problem for a strongly degenerate quasi-linear parabolic equation with an application to a model of pressure filtration. SIAM journal on mathematical analysis, 34(3) (2002:) 611-635.
	\bibitem{CSS2008} L. Caffarelli, S. Salsa, and L. Silvestre, Regularity estimates for the solution and the free boundary of the obstacle problem for the fractional Laplacian, Inventiones mathematicae 171.2 (2008): 425-461.
\bibitem{CPS2004} L. Caffarelli, A. Petrosyan, and H. Shahgholian, Regularity of a free boundary in parabolic potential theory, Journal of the American Mathematical Society 17.4 (2004): 827-869.	
	\bibitem{Ca1999} J. A. Carrillo, Entropy solutions for nonlinear degenerate problems. Archive for rational mechanics and analysis, 147(4) (1999): 269-361.
	\bibitem{COP2009} J. A. Carrillo, M. R. D'Orsogna, and V. Panferov, Double milling in self-propelled swarms from kinetic theory, Kinetic \& Related Models, 2 (2009), pp. 363-378.	
\bibitem{CSV2015} G- Q. Chen, H. Shahgholian, and J. L. Vazquez, Free boundary problems: the forefront of current and future developments. Philosophical Transactions of the Royal Society A: Mathematical, Physical and Engineering Sciences, 373(2050), 20140285 (2015).	
	\bibitem{CF2000} X. Chen and A. Friedman, A free boundary problem arising in a model of wound healing, SIAM Journal on Mathematical Analysis 32.4 (2000): 778-800.	
	\bibitem{ChenEtAl2021}
	L. Chen, A. Holzinger, A. J\"ungel, and N. Zamponi, Analysis and mean-field derivation of a porous-medium equation with fractional diffusion. ArXiv preprint arXiv:2109.08598 (2021).
	\bibitem{CK2005} Q. G. Chen and K. H. Karlsen, Quasilinear anisotropic degenerate parabolic equations with	time-space dependent diffusion coefficients, Comm. Pure and Applied Analysis (4) (2005): 241–266.	
	\bibitem{CP2003} Q. G. Chen and B. Perthame, Well-posedness for non-isotropic degenerate parabolic-hyperbolic equations, Ann. I. H. Poincare (20) (2003): 645–668.
	\bibitem{CBB1996} F. Concha, M. C. Bustos, and A. Barrientos, phenomenological theory of sedimentation. Advances in Fluid Mechanics, (1996).
	\bibitem{DS1983} E. DiBenedetto, and R. E. Showalter, A free-boundary problem for a degenerate parabolic system. WISCONSIN UNIV-MADISON MATHEMATICS RESEARCH CENTER, 1983.
\bibitem{D2013} Y. Du, Spreading profile and nonlinear Stefan problems, Bull. Inst. Math. Acad. Sin. 8 (2013): 413-430.
	\bibitem{DG2012} Y. Du and Z. Guo, The Stefan problem for the Fisher–KPP equation. Journal of Differential Equations 253.3 (2012): 996-1035.
	\bibitem{DMW2014} Y. Du, H. Matano, and K. Wang, Regularity and asymptotic behavior of nonlinear Stefan problems, Archive for Rational Mechanics and Analysis 212.3 (2014): 957-1010.
\bibitem{FS2015} A. Figalli and H. Shahgholian, A general class of free boundary problems for fully nonlinear parabolic equations. Annali di Matematica Pura ed Applicata (1923-) 194.4 (2015): 1123-1134.
	\bibitem{F1968} A. Friedman, The Stefan problem in several space variables. Transactions of the American Mathematical Society 133.1 (1968): 51-87.
	\bibitem{FriKam1980} A.~Friedman and S.~Kamin. The asymptotic behavior of gas in an $n$-dimensional porous medium. Transactions of the American Mathematical Society 262.2 (1980): 551-563.
	\bibitem{FK1975} A. Friedman, and D. Kinderlehrer, A one phase Stefan problem, Indiana University Mathematics Journal 24.11 (1975): 1005-1035.	
	\bibitem{GHSSV2014} S.~G\"ottlich, S.~Hoher, P.~Schindler, V.~Schleper, and A.~Verl, Modeling, simulation and validation of material flow on conveyor belts, Applied Mathematical Modelling, 38 (2014), pp. 3295-3313.	
	\bibitem{GKS2018} S.~G\"ottlich, S.~Knapp, and P.~Schillen, 
	A pedestrian flow model with stochastic
	velocities: Microscopic and macroscopic approaches, Kinetic and Related Models, 11 (2018), pp. 1333-1358.
	\bibitem{GKM2017} M. Graf, M. Kunzinger, and D. Mitrovic, Well-posedness theory for degenerate parabolic equations on Riemannian manifolds. Journal of Differential Equations 263(8) (2017): 4787-4825.	
	\bibitem{HM1995} D.~Helbing and P.~Moln\'ar, Social force model for pedestrian dynamics, Physical Review E, 51 (1995), pp. 4282-4286.
	\bibitem{HS2002} A. Henrot and H. Shahgholian, The one phase free boundary problem for the p-Laplacian with non-constant Bernoulli boundary condition. Transactions of the American Mathematical Society 354.6 (2002): 2399-2416.
	\bibitem{KN1978} D. Kinderlehrer and L. Nirenberg, The smoothness of the free boundary in the one phase Stefan problem, Communications on Pure and Applied Mathematics 31.3 (1978): 257-282.
	\bibitem{Lie1996} G.~M.~Lieberman. Second order parabolic differential equations. World scientific, 1996.
\bibitem{S1984} M. Struwe, On a free boundary problem for minimal surfaces, Inventiones mathematicae 75.3 (1984): 547-560.
	\bibitem{TKBB2003} E. M. Tory, K. H. Karlsen, R. Bürger, and S. Berres, Strongly degenerate parabolic-hyperbolic systems modeling polydisperse sedimentation with compression. SIAM Journal on Applied Mathematics, 64(1) (2003): 41-80.	
	\bibitem{Vol2000}A. I. Volpert, Generalized solutions of degenerate second-order quasilinear parabolic and elliptic equations. Advances in Differential Equations, 5(10-12) (2000): 1493-1518.
	\bibitem{WGA2021} J. Weissen, S. G\"ottlich, and D. Armbruster, Density dependent diffusion models for the interaction of particle ensembles with boundaries. Kinetic \& Related Models, 14(4) (2021): 681-704.  
	\bibitem{Wu1984} Z. Wu, A Free Boundary Problem for Degenerate Quasilinear Parabolic Equations. WISCONSIN UNIV-MADISON MATHEMATICS RESEARCH CENTER (1984).                    
		\bibitem{WZYL_Book} Z. Wu, J. Zhao, J. Yin, and H. Li, Nonlinear Diffusion Equations, World Scientific, Singapore, (2001).
	\bibitem{VH1969} A. I. Vol’pert and S. I. Hudjaev, Cauchy’s problem for degenerate second order quasilinear parabolic equations, Math. USSR-Sb. 7(3) (1969): 365-387.
	\bibitem{Z1989} J. Zhao, Uniqueness of solutions of quasilinear degenerate parabolic equations, Northeast. Math. J. (1) (1985): 153-165.
	\bibitem{Zhao1997} J. Zhao, Lipschitz continuity of the free boundary of some nonlinear degenerate parabolic equations. Nonlinear Analysis: Theory, Methods \& Applications, 28(6) (1997): 1047-1062.
	\bibitem{ZL1990} J. Zhao, and Li Yi, A free boundary problem for quasilinear degenerate parabolic equations with general degeneracy, Acta Mathematica Sinica 6.4 (1990): 364-382.
\end{thebibliography}
\end{document}